\pdfoutput=1
\RequirePackage{ifpdf}
\ifpdf 
\documentclass[pdftex]{sigma}
\else
\documentclass{sigma}
\fi

\usepackage{mathdots}\usepackage[all]{xy}

\newcommand{\Sym}{{\mathrm {Sym}}}
\newcommand{\C}{{\mathbb {C}}}
\newcommand{\R}{{\mathbb {R}}}
\newcommand{\N}{{\mathbb {N}}}
\newcommand{\Z}{{\mathbb {Z}}}
\newcommand{\HH}{{\mathbb {H}}}
\newcommand{\sSU}{{\mathsf {SU}}}
\newcommand{\sSL}{{\mathsf {SL}}}
\newcommand{\sSO}{{\mathsf {SO}}}
\newcommand{\sO}{{\mathsf {O}}}
\newcommand{\sSp}{{\mathsf {Sp}}}
\newcommand{\sSpin}{{\mathsf {Spin}}}
\newcommand{\sGL}{{\mathsf {GL}}}

\newcommand{\sG}{{\mathsf G}}

\newcommand{\sH}{{\mathsf H}}
\newcommand{\sS}{{\mathsf S}}

\newcommand{\sP}{{\mathsf P}}
\newcommand{\sL}{{\mathsf L}}

\newcommand{\sU}{{\mathsf U}}

\newcommand{\Teich}{{\mathsf{Teich}}}

\newcommand{\Fuch}{{\mathsf{Fuch}}}
\newcommand{\Isom}{{\mathsf{Isom}}}
\newcommand{\Bihol}{{\mathsf{Bihol}}}
\newcommand{\QFuch}{{\mathsf{QFuch}}}
\newcommand{\Hom}{{\mathrm{Hom}}}

\newcommand{\Pic}{{\mathsf{Pic}}}

\newcommand{\End}{{\mathrm{End}}}
\newcommand{\Hit}{{\mathsf{Hit}}}

\newcommand{\Id}{{\mathrm{Id}}}
\newcommand{\Ad}{{\operatorname{Ad}}}
\newcommand{\ad}{{\operatorname{ad}}}
\newcommand{\tr}{{\operatorname{tr}}}
\newcommand{\Xx}{{\mathcal{X}}}
\newcommand{\Bb}{{\mathcal{B}}}
\newcommand{\Oo}{{\mathcal{O}}}
\newcommand{\Kk}{{\mathcal{K}}}
\newcommand{\Hh}{{\mathcal{H}}}
\newcommand{\Ee}{{\mathcal{E}}}
\newcommand{\Gg}{{\mathcal{G}}}
\newcommand{\p}{{\partial}}
\newcommand{\Mm}{{\mathcal{M}}}
\newcommand{\Pp}{{\mathcal{P}}}
\newcommand{\fg}{{\mathfrak{g}}}
\newcommand{\fm}{{\mathfrak{m}}}
\newcommand{\fh}{{\mathfrak{h}}}
\newcommand{\fsl}{{\mathfrak{sl}}}
\newcommand{\fso}{{\mathfrak{so}}}
\newcommand{\fs}{{\mathfrak{s}}}
\newcommand{\rk}{\operatorname{rk}}

\newcommand{\mtrx}[1]{\left (\begin{matrix}#1\end{matrix}\right)}
\newcommand{\smtrx}[1]{\left (\begin{smallmatrix}#1\end{smallmatrix}\right)}

\numberwithin{equation}{section}

\newtheorem{Theorem}{Theorem}[section]
\newtheorem{Corollary}[Theorem]{Corollary}

\newtheorem{Question}[Theorem]{Question}
\newtheorem{Proposition}[Theorem]{Proposition}
{ \theoremstyle{definition}
\newtheorem{Definition}[Theorem]{Definition}
\newtheorem{Example}[Theorem]{Example}
\newtheorem{Remark}[Theorem]{Remark} }

\begin{document}

\allowdisplaybreaks

\newcommand{\arXivNumber}{1809.06786}

\renewcommand{\thefootnote}{}

\renewcommand{\PaperNumber}{010}

\FirstPageHeading

\ShortArticleName{Studying Deformations of Fuchsian Representations with Higgs Bundles}

\ArticleName{Studying Deformations of Fuchsian Representations\\ with Higgs Bundles\footnote{This paper is a~contribution to the Special Issue on Geometry and Physics of Hitchin Systems. The full collection is available at \href{https://www.emis.de/journals/SIGMA/hitchin-systems.html}{https://www.emis.de/journals/SIGMA/hitchin-systems.html}}}

\Author{Brian COLLIER}

\AuthorNameForHeading{B.~Collier}

\Address{Department of Mathematics, University of Maryland, College Park, MD 20742, USA}
\Email{\href{mailto:briancollier01@gmail.com}{briancollier01@gmail.com}}
\URLaddress{\url{https://www.math.umd.edu/~bcollie2/}}

\ArticleDates{Received October 16, 2018, in final form February 02, 2019; Published online February 12, 2019}

\Abstract{This is a survey article whose main goal is to explain how many components of the character variety of a closed surface are either deformation spaces of representations into the maximal compact subgroup or deformation spaces of certain Fuchsian representations. This latter family is of particular interest and is related to the field of higher Teichm\"uller theory. Our main tool is the theory of Higgs bundles. We try to develop the general theory of Higgs bundles for real groups and indicate where subtleties arise. However, the main emphasis is placed on concrete examples which are our motivating objects. In particular, we do not prove any of the foundational theorems, rather we state them and show how they can be used to prove interesting statements about components of the character variety. We have also not spent any time developing the tools (harmonic maps) which define the bridge between Higgs bundles and the character variety. For this side of the story we refer the reader to the survey article of Q.~Li~[arXiv:1809.05747].}


\Classification{14D20; 14F45; 14H60}

\renewcommand{\thefootnote}{\arabic{footnote}}
\setcounter{footnote}{0}

\section{An introduction to the character variety}
Let $S$ be a closed oriented surface of genus $g\geq 2$. Denote the fundamental group of $S$ by $\Gamma$, and recall that $\Gamma$ has the standard presentation
\begin{gather*} \Gamma=\left\langle a_1,\dots,a_g,b_1,\dots,b_g\,|\,\prod\limits_{j=1}^g[a_j,b_j]=1\right\rangle.\end{gather*}
Fix also a real reductive Lie group $\sG$. For example $\sG$ could be one of the following groups
\begin{gather*} \xymatrix@=1em{\sGL(n,\R),&\sGL(n,\C),&\sSL(n,\C),&\sSp(2n,\R),&\sP\sSL(n,\R)=\sSL(n,\R)/\pm\Id,\\\sU(n),&\sSO(n),&\sSU(p,q),&\sSO(p,q),&\sS(\sO(p)\times \sO(q)),}\end{gather*}
but $\sG$ \emph{cannot} be a group like
\begin{gather*} \sP=\big\{\smtrx{a&b\\0&c}\in\sGL(2,\C)\big\}.\end{gather*}
One main property of a reductive Lie group $\sG$ is that, up to conjugation, there is a unique maximal compact subgroup $\sH< \sG$. We will heavily use this property. In fact there is a~homotopy equivalence
\begin{gather*} \sH\simeq\sG.\end{gather*}

A group homomorphism $\rho\colon \Gamma\to\sG$ will be referred to as a \emph{representation}. Consider the set $\Hom(\Gamma,\sG)$ of all such representations. This set has a topological structure induced by the inclusion
\begin{gather*} \xymatrix@R=0em{\Hom(\Gamma,\sG)\ar[r]&\sG^{2g},\\\rho\ar@{|->}[r]&(\rho(a_1),\dots,\rho(b_g)).}\end{gather*}
Since a representation $\rho$ must satisfy $\prod\limits_{j=1}^g[\rho(a_j),\rho(b_j)]=\Id$, the image consists of $2g$-tuples of elements of $\sG$ which satisfy this relation. For semisimple groups, the dimension of $\Hom(\Gamma,\sG)$ was computed in \cite{SymplecticNatureofFund} to be
\begin{gather*}\dim(\Hom(\Gamma,\sG))=2g\cdot\dim\sG-\dim\sG.\end{gather*}

The group $\sG$ acts on $\Hom(\Gamma,\sG)$ by conjugation: for $\gamma\in\Gamma$ and $g\in\sG$,
\begin{gather*} (\rho\cdot g)(\gamma)=g^{-1}\rho(\gamma)g.\end{gather*}
The quotient space $\Hom(\Gamma,\sG)/\sG$ consists of conjugacy classes of representations. Unless $\sG$ is compact, $\Hom(\Gamma,\sG)/\sG$ is not Hausdorff.

\begin{Example} Let $\rho_1(\gamma)=\Id$ for all $\gamma\in \Gamma$ and define $\rho_2$ by $\rho_2(a_j)=\Id$ for all $j$ and
\begin{gather*}\rho_2(b_j)=\begin{cases}
 \Id,&j\neq g,\\
 \smtrx{1&1\\0&1},&j=g.
 \end{cases}\end{gather*}
If $g_t=\smtrx{t&0\\0&t^{-1}}$, then $g_t^{-1}\rho_2(a_j)g_t=\Id$ for all $j$ and
\begin{gather*} g_t^{-1}\rho_2(b_j)g_t=\begin{cases}
 \Id,&j\neq g,\\
 \smtrx{1&t^{-2}\\0&1},&j=g.
 \end{cases}\end{gather*}
 The $\sG$-orbits through $\rho_1$ and $\rho_2$ are disjoint in $\Hom(\Gamma,\sG)$, but
\begin{gather*} \rho_1\in\overline{\rho_2\cdot G}.\end{gather*}
 Thus, $[\rho_1]$ and $[\rho_2]$ are not separable in $\Hom(\Gamma,\sG)/\sG$.
\end{Example}

To get a Hausdorff quotient we restrict to a subset of $\Hom(\Gamma,\sG)$. A representation $\rho\colon \Gamma\to \sG$ is called \emph{reductive} if the composition with the adjoint representation
\begin{gather*} \xymatrix{\Gamma\ar[r]^\rho&\sG\ar[r]^{\Ad}&\sGL(\fg)}\end{gather*}
decomposes as a direct sum of irreducible representations. This is equivalent to the $\sG$-orbit through $\rho$ being closed in $\Hom(\Gamma,\sG)$. When $\sG$ is a complex algebraic group, this is also equivalent to the Zariski closure of $\rho(\Gamma)$ being a reductive subgroup.

Denote the set of reductive representations by $\Hom^+(\Gamma,\sG)$. The \emph{$\sG$-character variety of $\Gamma$} is defined to be the Hausdorff space
\begin{gather*}\Xx(\Gamma,\sG)=\Hom^+(\Gamma,\sG)/\sG.\end{gather*}
For semisimple groups, the dimension of $\Xx(\Gamma,\sG)$ is $(2g-2)\dim(\sG)$ \cite{SymplecticNatureofFund}.
\begin{Remark}When $\sG$ is complex algebraic, one gets the same spaces as the GIT-quotient.
\end{Remark}

\subsection{Fuchsian representations} Important examples of points in a character variety come from Fuchsian representations and are related to hyperbolic metrics on $S$. These examples will play a fundamental role throughout the article.

Let $X$ be a Riemann surface structure on $S$. By the uniformization theorem,
\begin{gather*}X=\HH^2/\rho_X(\Gamma),\end{gather*}
where $\HH^2$ is the upper half-plane and $\rho_X\colon \Gamma\hookrightarrow\Bihol(\HH^2)\cong\sP\sSL(2,\R)$. By classical results of Riemann, there are $3g-3$ complex parameters for the deformations of the Riemann surface~$X$. Such deformations give rise to a $(6g-6)$-real dimensional subset of $\Xx(\Gamma,\sP\sSL(2,\R))$. Moreover, such deformations define an open neighborhood of $\rho_X\in\Xx(\Gamma,\sP\sSL(2,\R))$ since $6g-6$ is the dimension of $\Xx(\Gamma,\sP\sSL(2,\R))$. This leads us to define the set of \emph{Fuchsian representations}:
\begin{gather*}\Fuch(\Gamma)=\{\rho\colon \Gamma\to\sP\sSL(2,\R)\,|\,\text{discrete and faithful}\}/\sP\sSL(2,\R).\end{gather*}

Since being discrete and faithful is a closed condition in $\Xx(\Gamma,\sG)$, $\Fuch(\Gamma)$ is a closed subset. By the above argument, $\Fuch(\Gamma)$ is also open. Thus, $\Fuch(\Gamma)\subset\Xx(\Gamma,\sP\sSL(2,\R))$ is an open and closed subset which is identified with the Teichm\"uller space of the oriented surface $S$ and the surface $\bar S$ with the opposite orientation
\begin{gather*}\Fuch(\Gamma)\cong\Teich(S)\sqcup\Teich\big(\bar S\big).\end{gather*}
The Teichm\"uller space of an oriented surface $S$ is defined to be the space of isotopy classes of (marked) Riemann surface structures on $S$. This space is famously an open cell of real dimension $6g-6$.

The group $\sP\sSL(2,\R)$ is also the orientation preserving isometry group of the hyperbolic plane. Thus, $\Fuch(\Gamma)$ also parameterizes the set of isotopy classes of hyperbolic metrics on the surface~$S$ and~$\bar S$.

\begin{Remark}The two connected components of $\Fuch(\Gamma)\subset\Xx(\Gamma,\sP\sSL(2,\R))$ arise because elements of $\sP\sSL(2,\R)$ preserve the orientation of $\HH^2$. The group $\sP\sGL(2,\R)$ is the full isometry group of the hyperbolic plane. Thus, if we consider $\Fuch(\Gamma)$ as a subset of the $\sP\sGL(2,\R)$-character variety, then there is only one component. Similarly, if we consider $\Fuch(\Gamma)$ as a subset of $\sP\sSL(2,\C)$-character variety, then it is a closed connected subset which is no longer open.
\end{Remark}

\subsection{Connected components}
One fundamental problem is to determine how many connected components a character variety has. Surprisingly, this question has not been answered in full generality. There is a topological invariant which helps distinguish connected components.

Denote the set of isomorphism classes of topological principal $\sG$-bundles on $S$ by $\Bb_{\sG}(S)$, this is the set of homotopy classes of maps from $S$ to the classifying space of $\sG$. For connected groups we have $\Bb_\sG(S)\cong H^2(S,\pi_1\sG)$. Every representation $\rho\colon \Gamma\to\sG$ defines a principal $\sG$-bundle $E_\rho\to S$
\begin{gather*}E_\rho=\big(\tilde S\times\sG\big)/\Gamma,\end{gather*}
where $\Gamma$ acts on its universal cover $\tilde S$ by deck transformations and by multiplication by~$\rho(\Gamma)$ on~$\sG$. Thus, we have a map $\Hom(\Gamma,\sG)\to\Bb_{\sG}(S)$. Moreover, this map is continuous and descends to a map
\begin{gather}\label{eq top invariant map}
\tau\colon \ \xymatrix{\pi_0(\Xx(\Gamma,\sG))\ar[r]&\Bb_\sG(S).}
\end{gather}
If $\Xx^\omega(\Gamma,\sG)=\tau^{-1}(\omega)$, then $\Xx(\Gamma,\sG)$ decomposes as
\begin{gather*}\coprod\limits_{\omega\in\Bb_\sG(S)}\Xx^\omega(\Gamma,\sG).\end{gather*}

\begin{Question}\label{question comp}When is the map $\tau$ from \eqref{eq top invariant map} injective? In other words, when does the topological invariant distinguish the connected components of the character variety?
\end{Question}
\begin{Remark}Note that when $\tau$ is injective, the question of the connected component count is not very interesting. We will mainly be interested in when $\tau$ is \emph{not} injective and understanding the special features of these components.
\end{Remark}
Question \ref{question comp} has been answered for many groups, but is open in general. For compact groups the map~$\tau$ is injective, this was proven by Narasimhan--Seshadri \cite{NarasimhanSeshadri} for $\sG=\sU(n)$ and Ramanathan \cite{ramanathan_1975} in general.

\begin{Theorem}\label{thm tau injective compact}
 If $\sG$ is compact $($i.e., $\sG=\sH)$, then $\tau$ is injective. Furthermore, if $\sG$ is also semisimple, then $\tau$ is a bijection.
\end{Theorem}

Since $\sH$ and $\sG$ are homotopic, we have $\Bb_\sG(S)=\Bb_\sH(S)$. Moreover, for each $\omega\in\Bb_\sH(S)$
\begin{gather*}\Xx^\omega(\Gamma,\sH)\subset\Xx^\omega(\Gamma,\sG).\end{gather*}
For complex groups, the map $\tau$ is also injective. This was proven for by J.~Li~\cite{JunLiConnectedness} for semisimple groups and Garc\'{\i}a-Prada and Oliveira \cite{Oliveira_GarciaPrada_2016} for reductive groups.
\begin{Theorem}\label{thm tau injective complex}
If $\sG$ is complex $($i.e., $\sG=\sH^\C)$, then $\tau$ is injective.
\end{Theorem}
We will prove these results using Higgs bundles in Section~\ref{section some comp results}. For $\sG$ a semisimple complex Lie group, the following corollary follows immediately from the two above theorems. It holds in general for complex reductive groups.
\begin{Corollary}\label{cor: complex deform to compact}For $\sG$ a complex reductive Lie group, every representation $\rho\colon \Gamma\to\sG$ can be continuously deformed to a compact representation $\Gamma\to\sH\hookrightarrow \sG$.
\end{Corollary}
The above corollary says that, for complex groups, the connected components of the character variety are not interesting. For real groups, the situation is more subtle.

\begin{Example}For $\sG=\sP\sSL(2,\R)$, the maximal compact subgroup is $\sH\cong\sSO(2)$. Since a circle bundle on a closed surface is determined by its degree, we have $\Bb_\sH(S)\cong\Z$. Thus,
\begin{gather*}\Xx(\Gamma,\sP\sSL(2,\R))=\coprod\limits_{d\in\Z}\Xx^d(\Gamma,\sP\sSL(2,\R)).\end{gather*}
However, the space $\Xx^d(\Gamma,\sP\sSL(2,\R))$ is empty when $|d|>2g-2$ \cite{MilnorMWinequality}. Moreover, when $|d|\leq 2g-2$, the space $\Xx^d(\Gamma,\sP\sSL(2,\R))$ is nonempty and connected~\cite{TopologicalComponents}. We will prove these statements using Higgs bundles in Section~\ref{section SO1q and Hitchin comp}.

For a $\sP\sSL(2,\R)$ representation $\rho$, the integer invariant can be defined as follows. Pick any $\rho$-equivariant map $f_\rho\colon \widetilde S\to \sP\sSL(2,\R)/\sH\cong\HH^2$. Such maps always exist since $\HH^2$ is contractible. We have a principal $\sH$-bundle $\sP\sSL(2,\R)\to \HH^2$. By equivariance, the pullback bundle $f_\rho^*\sP\sSL(2,\R))$ descends to a circle bundle on $S$. Define $\tau$ to be minus the degree of this bundle:
\begin{gather*}\tau(\rho)=-\deg((f_\rho^*\sP\sSL(2,\R))/\rho(\Gamma)).\end{gather*}
The bundle $\sP\sSL(2,\R)\to\HH^2$ is identified with the unit tangent bundle of $\HH^2$.
If $\rho$ is a Fuchsian representation, then we may choose $f_\rho$ to be the equivariant diffeomorphism uniformizing the Riemann surface $\HH^2/\rho(\Gamma)$. In this case, $\tau$ is given by the degree of the cotangent bundle. Namely, $\tau=2g-2$.
Thus, we have
\begin{gather*}\Fuch(\Gamma)\subset\Xx^{2g-2}(\Gamma,\sP\sSL(2,\R))\sqcup\Xx^{2-2g}(\Gamma,\sP\sSL(2,\R)).\end{gather*}
Since $\Fuch(\Gamma)$ is open and closed and $\Xx^{2g-2}(\Gamma,\sP\sSL(2,\R))$ is connected, the above inclusion is an equality.
\end{Example}

\begin{Example}
For $\sG=\sP\sSL(n,\R)$ the maximal compact subgroup is $\sSO(n)$ if $n$ is odd and $\sSO(n)/\pm\Id$ when $n$ is even. In this case,
\begin{gather}
\label{eq top PSLn bundles}\Bb_{\sP\sSL(n,\R)}(S)\cong H^2(S,\pi_1\sH)\cong \begin{cases}
\Z&\text{if $n=2$,}\\
\Z_2&\text{if $n=2k+1$,}\\
\Z_2\oplus \Z_2&\text{if $n=4k$,}\\
\Z_4&\text{if $n=4k+2$.}
\end{cases}
\end{gather}
For $n=2k+1$, the invariant $\omega\in H^2(S,\Z_2)$ is the second Stiefel--Whitney class of the $\sSO(n)$ bundle.
\end{Example}

\begin{Example}\label{Ex top types SOpq}
For $\sG=\sSO(p,q)$, the maximal compact subgroup is $\sS(\sO(p)\times \sO(q))$. We have
\begin{gather*}\Bb_{\sSO(p,q)}(S)\cong\begin{cases}
H^1(S,\Z_2)\cong\Z_2^{2g}&\text{if $p=q=1$},\\
H^1(S,\Z_2)\times H^2(S,\Z_2)\cong\Z_2^{2g+1}&\text{if $p=1$ and $2<q$},\\
H^1(S,\Z_2)\times H^2(S,\Z_2)\times H^2(S,\Z_2)\cong\Z_2^{2g+2}&\text{if $2<p\leq q$}.
\end{cases}
\end{gather*}
In the above cases, the element of $H^1(S,\Z_2)$ is the first Stiefel--Whitney class of an orthogonal bundle and each element of $H^2(S,\Z_2)$ is the second Stiefel--Whitney class of an orthogonal bundle.
The case of $p=2$ or $q=2$ is slightly more complicated.
\end{Example}

\section{Deforming Fuchsian representations}\label{sec deforming Fuch}
We have seen that the Teichm\"uller space of the surface $S$ is identified with the connected components $\Fuch(\Gamma)$ of the $\sP\sSL(2,\R)$-character variety. As a result, the representations in this component have special geometric significance. Given an embedding $\iota\colon \sP\sSL(2,\R)\to \sG$, we have
\begin{gather*}\iota(\Fuch(\Gamma))\subset \Xx(\Gamma,\sG).\end{gather*}
\begin{Question}
What is the deformation space of $\iota(\Fuch(\Gamma))\subset \Xx(\Gamma,\sG)$ like?
\end{Question}

Here we use the word deformation space of $\iota(\Fuch(\Gamma))$ to mean the connected component of $\Xx(\Gamma,\sG)$ which contains $\iota(\Fuch(\Gamma))$. There are many examples of interesting embeddings of $\sP\sSL(2,\R)$ into other Lie groups, below we discuss some particular interesting ones.
\subsection{Isometry groups of hyperbolic spaces} There is a natural embedding into $\Isom^+(\HH^n)\cong\sSO_0(1,n)$ given by
\begin{gather}\label{eq inclusion into SO(1,n)}
\xymatrix{\sP\sSL(2,\R)\cong\Isom^+(\HH^2)\cong\sSO_0(1,2)\ar[r]^{\ \ \ \ \ \ \iota_{1,n}}& \sSO_0(1,n)\cong\Isom^+(\HH^n).}
\end{gather}
This gives $\iota_{1,n}(\Fuch(\Gamma))\subset\Xx(\Gamma,\sSO_0(1,n))$, and small deformations of these representations are holonomies of complete hyperbolic $n$-manifolds called quasi-Fuchsian manifolds.

\begin{Definition}\label{def quasi fuchsian}A representation $\rho\colon \Gamma\to \Isom(\HH^n)$ is called \emph{convex cocompact} if it is discrete and faithful and $\rho(\Gamma)$ acts cocompactly (properly discontinuously with compact quotient) on a~convex domain in $\HH^n$. The set of quasi-Fuchsian representations $\QFuch(\Gamma)\subset\Xx(\Gamma,\sSO_0(1,n))$ is defined to be
\begin{gather*}\QFuch(\Gamma)=\{[\rho]\in\Xx(\Gamma,\sSO_0(1,n))\,|\, \text{$\rho$ is convex cocompact}\}.\end{gather*}
\end{Definition}
\begin{Remark}If $[\rho]\in\QFuch(\Gamma)$, then can be deformed to $\iota_{1,n}(\Fuch(\Gamma))$. Moreover, the set of convex cocompact representations is open in $\Xx(\Gamma,\sSO_0(1,n))$. Thus, any sufficiently small deformation of a Fuchsian representation in $\iota_{1,n}(\Fuch(\Gamma))$ is quasi-Fuchsian. However, unlike $\Fuch(\Gamma)\subset\Xx(\Gamma,\sSO_0(1,2))$, the set $\QFuch(\Gamma)\subset\Xx(\Gamma,\sSO_0(1,n))$ is \emph{not closed}. Namely, there families of convex cocompact representations whose limit is discrete and faithful, but not convex cocompact.
\end{Remark}
In fact, we have the following:

\begin{Proposition}\label{prop deforming quasi fuchsian to compact}
For $n>2$, any representation $\rho\in\iota_{1,n}(\Fuch(\Gamma))$ can be continuously deformed to a compact representation.
\end{Proposition}
\begin{proof}Note that it suffices to prove the statement for $\sSO_0(1,3)$. Recall that there is an isomorphism of Lie groups $\sSO_0(1,3)\cong\sP\sSL(2,\C)$. The result now follows from Corollary~\ref{cor: complex deform to compact}.
\end{proof}

\begin{Remark}Another interesting embedding is given by the isomorphism
\begin{gather*}\sP\sSL(2,\R)\cong\sP\sU(1,1)\end{gather*}
and the embedding $\sP\sU(1,1)\to \sP\sU(1,n)\cong\Isom(\C\HH^n)$ into the isometry group of the complex hyperbolic space. Deformations of $\Fuch(\Gamma)\subset\Xx(\Gamma,\sP\sU(1,n))$ under this embedding satisfy a~rigidity phenomenon \cite{GoldmanMillsonRigidity,ToledoRigidityU1q}. This is a special case of the more general situation of maximal representations into a Hermitian Lie group of non-tube type (see for example~\cite{BIWmaximalToledoAnnals}). We will not discuss this situation further.
\end{Remark}

\subsection{Principal embedding}
Recall that for each dimension $n$, there is a unique irreducible representation
\begin{gather*}\iota_{\rm pr}\colon \ \sSL(2,\R)\to \sSL\big(\R^n\big),\end{gather*} which is given by the $(n-1)^{\rm st}$-symmetric product of the standard representation. Moreover, it is straight forward to check that this induces an embedding
\begin{gather*}\iota_{\rm pr}\colon \ \sP\sSL(2,\R)\to\sP\sSL(n,\R).\end{gather*}
We will call this embedding the \emph{principal embedding}.

More generally, if $\sG$ is a split real Lie group of adjoint type, there is a unique preferred (principal) embedding
\begin{gather}\label{eq principal embedding}\iota_{\rm pr}\colon \ \sP\sSL(2,\R)\to\sG.
\end{gather} We will not go into the Lie theory necessary to define the principal embedding in general, see~\cite{ptds} for more details on the general setup. We will explicitly describe $\iota_{\rm pr}$ for the classical groups.
\begin{Example}\label{ex SO(p,p+1) principal}When $n=2p+1$ the principal embedding $\iota_{\rm pr}\colon \sSL(2,\R)\to\sSL\big(\R^{2p+1}\big)$ preserves a symmetric nondegenerate quadratic form of signature $(p,p+1)$. Thus, we have an embedding
\begin{gather*}\iota_{\rm pr}\colon \ \sP\sSL(2,\R)\to\sSO_0(p,p+1)\subset\sP\sSL(2p+1,\R).\end{gather*}
This is the principal embedding~\eqref{eq principal embedding} for the split group $\sG=\sSO_0(p,p+1)$.

Similarly, when $n=2p$ the principal embedding $\iota_{\rm pr}\colon \sSL(2,\R)\to\sSL(\R^{2p})$ preserves a nondegenerate symplectic form. Thus, we have an embedding
\begin{gather*}\iota_{\rm pr}\colon \ \sP\sSL(2,\R)\to\sP\sSp(2p,\R)\subset\sP\sSL(2p,\R).\end{gather*}
This is the principal embedding \eqref{eq principal embedding} for the split group adjoint $\sG=\sP\sSp(2p,\R)$.
\end{Example}
The deformation space of $\iota_{\rm pr}(\Fuch(\Gamma))\subset\Xx(\Gamma,\sG)$ is called the \emph{Hitchin component} or \emph{Hitchin components}.
\begin{Definition}\label{Def Hit comp}
Let $\sG$ be a simple split real Lie group of adjoint type, a Hitchin component
\begin{gather*}\Hit(\sG)\subset\Xx(\Gamma,\sG)\end{gather*}
is a connected component containing a component of $\iota_{\rm pr}(\Fuch(\Gamma))$.
\end{Definition}
Unlike the embedding $\iota_{1,n}\colon \sP\sSL(2,\R)\to\sSO_0(1,n)$, representations in $\Hit(\sG)$ cannot be deformed to compact representations.

\begin{Theorem}[Hitchin \cite{liegroupsteichmuller}] If $\rho\in\Hit(\sG)$, then $\rho$ cannot be deformed to a compact representation. In particular,
\begin{gather*}|\pi_0(\Xx(\Gamma,\sG))|\geq 1+|\pi_0(\Xx(\Gamma,\sH))|.\end{gather*}
\end{Theorem}

\begin{Remark}In \cite{AnosovFlowsLabourie}, Labourie showed that all representations in a Hitchin component satisfy a certain dynamical property called the Anosov property which generalizes the notion of convex cocompactness to higher rank Lie groups. As a consequence, every representation in a~Hitchin component is discrete and faithful. Moreover, like $\Fuch(\Gamma)$, representations in a~Hitchin component are holonomies of certain geometric structures on compact manifolds~\cite{guichard_wienhard_2012}. Since $\Hit(\sG)$ shares many features with the Teichm\"uller space of~$S$, it has been called a \emph{higher Teichm\"uller component} (see for example~\cite{BurgerIozziWienhardSurvey} and~\cite{InvitationToHigherTeich}). We will not discuss this perspective any more, however the components discussed in this article which are deformation spaces of Fuchsian representations are intimately related with the field of higher Teichm\"uller theory.
\end{Remark}

For the group $\sP\sSL(n,\R)$, Hitchin also proved that there are no other components.
\begin{Theorem}[Hitchin \cite{liegroupsteichmuller}] For $n>2$ we have
\begin{gather*}|\pi_0(\Xx(\Gamma,\sP\sSL(n,\R))|=\begin{cases}
3&\text{if $n$ is odd},\\
6&\text{if $n$ is even}.
\end{cases}\end{gather*}
\end{Theorem}
\begin{Remark}\label{Remark 2 Hitchin comp PSL2n}Recall that $\Fuch(\Gamma)\subset\Xx(\Gamma,\sP\sSL(2,\R))$ has two connected components, however, these components are isomorphic via an outer automorphism of $\sP\sSL(2,\R)$. The number of Hitchin components $\Hit(\sP\sSL(n,\R))$ depends on the parity of $n$. Namely, the map $\iota_{\rm pr}\colon \Fuch(\Gamma)\to\Xx(\Gamma,\sP\sSL(n,\R))$ is $2:1$ when $n$ is odd and injective when $n$ is even. There are thus two Hitchin components for $\sG=\sP\sSL(2n,\R)$ and one Hitchin component for $\sG=\sP\sSL(2n+1,\R)$.
\end{Remark}

\begin{Corollary}\label{cor dichotomy hitchin}If $\rho\in\Xx(\Gamma,\sP\sSL(n,\R))$, then there is a dichotomy: either $\rho$ can be deformed to compact representation or $\rho$ can be deformed to a Fuchsian representation in $\iota_{\rm pr}(\Fuch(\Gamma))$.
\end{Corollary}

\begin{Remark}A generalization of the embedding \eqref{eq inclusion into SO(1,n)} is given by
\begin{gather}\label{eq pq embedd}
\iota_{p,q}\colon \ \sSO(p,p-1)\to \sSO(p,q).
\end{gather}
The embedding
\begin{gather*}\xymatrix{\sP\sSL(2,\R)\ar[r]^{\iota_{\rm pr}}&\sSO(p,p-1)\ar[r]^{\iota_{p,q}}&\sSO(p,q)} \end{gather*}
will play an important role in Theorem~\ref{thm dichotomy pq}. In fact, when $q=p$ the principal embedding $\iota_{\rm pr}\colon \sP\sSL(2,\R)\to\sSO(p,p)$ is given by the principal embedding into $\sSO(p,p-1)$ followed by $\iota_{p,p}$
\begin{gather*}\xymatrix{\sP\sSL(2,\R)\ar[r]^{\iota_{\rm pr}}\ar@/_1pc/[rr]_{\iota_{\rm pr}}&\sSO(p,p-1)\ar[r]^{\iota_{p,p}}&\sSO(p,p).}\end{gather*}
\end{Remark}

\section{Higgs bundles}
We now shift our focus to a moduli space of holomorphic objects on a Riemann surface called Higgs bundles. Roughly, a~Higgs bundle is a holomorphic bundle with some extra data, and the moduli space parameterizes isomorphism classes of Higgs bundles which are called polystable. This theory was developed by Hitchin \cite{selfduality,IntSystemFibration} and Simpson~\cite{SimpsonVHS,localsystems}. At first glance, Higgs bundles and surface group representations seem to have little to do with each other. However, a~remarkable theorem, known as the \emph{nonabelian Hodge correspondence}, gives a homeomorphism between the two moduli spaces. Higgs bundles thus provide a powerful tool for addressing certain questions about the topology of the character variety.

\begin{Theorem}[nonabelian Hodge correspondence]\label{Thm: Nonabelian Hodge Correspondence}
Let $S$ be a closed orientable surface of genus at least two. For each Riemann surface structure $X$ on $S$, the moduli space of $\sG$-Higgs bundles on $X$ is homeomorphic to the character variety $\Xx(\pi_1(S),\sG)$. Moreover, the smooth loci of each space are diffeomorphic.
\end{Theorem}

\begin{Remark}\label{herm metric remark}One direction of the nonabelian Hodge correspondence asserts that, for each polystable $\sG$-Higgs bundle, there is a special metric which can be used to construct a flat $\sG$-connection. For principal bundles a metric is by definition a reduction of structure group to the maximal compact subgroup. The other direction asserts that, for each reductive representation and each choice of Riemann surface structure $X$ on $S$, there is an equivariant harmonic map $\widetilde X\to\sG/\sH$ from the universal cover to the Riemannian symmetric space. From such a map one constructs a polystable $\sG$-Higgs bundle. For more details on the correspondence, we refer the reader to Q. Li's survey article~\cite{QionglingMiniCourseNotes}.
\end{Remark}

\subsection{Definitions}
As before, let $\sG$ be a reductive Lie group with maximal compact $\sH$ and Cartan decomposition $\fg=\fh\oplus\fm$. Complexifying gives an $\Ad_{\sH^\C}$-invariant decomposition $\fg^\C=\fh^\C\oplus\fm^\C$. Fix a compact Riemann surface $X$ with genus $g\geq2$ and let $K$ denote its holomorphic cotangent bundle. Given a principal $\sH^\C$-bundle $P$, let $P\big[\fm^\C\big]$ denote the associated bundle with fiber $\fm^\C$:
\begin{gather*}P\big[\fm^\C\big]=\big(P\times \fm^\C\big)/{\sim},\end{gather*}
where $(p\cdot h,v)\sim(p,\Ad_{h^{-1}}v)$ for $h\in\sH^\C$.

\begin{Definition}A $\sG$-Higgs bundle on $X$ is a pair $(\Pp,\varphi)$ where\samepage
\begin{itemize}\itemsep=0pt
\item $\Pp\to X$ is a holomorphic principal $\sH^\C$-bundle and
\item $\varphi$ is a holomorphic section of the associated bundle $\Pp\big[\fm^\C\big]\otimes K$.
\end{itemize}
The holomorphic section $\varphi$ is called the \emph{Higgs field}.
\end{Definition}

\begin{Example}
If $\sG$ is compact, then $\sG=\sH$ and $\fm=\{0\}$. In this case, a $\sG$-Higgs bundle is just a holomorphic principal $\sH^\C$-bundle. So, for compact groups, the moduli space $\Mm(\sH)$ of Higgs bundles is identical to the moduli space of holomorphic $\sH^\C$-bundles.
\end{Example}

\begin{Example}
If $\sG$ is complex, then $\fg=\fh^\C$ and $\fg^\C=\fh^\C\oplus\fm^\C\cong\fg\oplus\fg$. In this case, a $\sG$-Higgs bundles is a pair $(\Pp,\varphi)$, where $\Pp$ is a holomorphic $\sG$-bundle and $\varphi$ is a holomorphic section of the adjoint bundle $\Pp[\fg]$ twisted by $K$.
\end{Example}

Rather than working with principal bundles, we will usually pick a faithful linear representation of $\sG^\C$ and work with vector bundles. A~faithful representation $\sG^\C\to\sGL(V)$ defines a~representation $\beta\colon \sH^\C\to\sGL(V)$ and an embedding $\fm^\C\hookrightarrow \End(V)$. With this data fixed, a~$\sG$-Higgs bundle $(\Pp,\varphi)$ gives rise to a pair $(E,\Phi)$, where $E\to X$ is the holomorphic vector bundle~$\Pp[V]$ and $\Phi\in H^0(\End(E)\otimes K)$ is given by $\varphi$ under the inclusion $\Pp\big[\fm^\C\big]\otimes K\hookrightarrow \End(\Pp[V])\otimes K$.
\begin{Example}When $\sG=\sSL(n,\C)$ we take $\sG\to\sGL(\C^n)$ to be the standard representation. An $\sSL(n,\C)$-Higgs bundle thus defines a pair $(E,\Phi)$, where $E\to X$ is a holomorphic rank $n$ vector bundle and $\Phi\in H^0(\End(E)\otimes K)$ satisfies $\tr(\Phi)=0$. Moreover, the standard volume form on $\C^n$ is preserved by the standard representation of $\sSL(n,\C)$, and a holomorphic principal $\sSL(n,\C)$-bundle is equivalent to a holomorphic vector bundle $E$ equipped with a holomorphic volume form $\omega\in H^0(\Lambda^nE)$. Thus, an $\sSL(n,\C)$-Higgs bundle is equivalent to a triple $(E,\omega,\Phi)$. Note that the holomorphic volume form $\omega$ is equivalent to a holomorphic trivialization of the determinant line bundle~$\Lambda^n E$. We will usually suppress $\omega$ from the notation.
\end{Example}

\begin{Example}For $\sG=\sSL(n,\R)$ we have $\sH=\sSO(n)$ and the Cartan decomposition is given by
\begin{gather*}\fsl(n,\R)\cong\fso(n)\oplus \mathrm{Sym}_0\big(\R^n\big),\end{gather*}
where $\mathrm{Sym}_0\big(\R^n\big)$ is the vector space of traceless symmetric matrices. Again, using the standard representation of $\sSL(n,\C)$, we see that an $\sSL(n,\R)$-Higgs bundle gives rise to a triple $(E,\omega,\Phi)$ as in the previous example.

Since $\sH^\C=\sSO(n,\C)$, the restriction of the standard representation of $\sSO(n,\C)$ preserves a nondegenerate symmetric complex bilinear form on $\C^n$, a holomorphic principal $\sSO(n,\C)$-bundle\footnote{For $\sO(n,\C)$ a holomorphic principal bundle is equivalent to a pair $(E,Q_E)$.} is equivalent to a triple $(E,\omega,Q_E)$ where $Q_E\in H^0(S^2(E)\otimes K)$ is everywhere nondegenerate. Equivalently, $Q_E$ defines a symmetric holomorphic isomorphism $Q_E\colon E\to E^*$.
Since $\fm^\C=\mathrm{Sym}_0(\C^n)$, the Higgs field $\Phi$ is symmetric with respect to the quadratic form $Q_E$, i.e.,
\begin{gather*}\Phi^TQ_E=Q_E\Phi.\end{gather*}
An $\sSL(n,\R)$-Higgs bundle is thus equivalent to a tuple $(E,\omega,Q_E,\Phi)$.
\end{Example}

\begin{Example}For $\sG=\sSO(p,q)$, we have $\sH=\sS(\sO(p)\times \sO(q))$ and $\fh=\fso(p)\oplus\fso(q)$. With respect to a splitting $\R^{p+q}=\R^p\oplus\R^q$ we may decompose a matrix $X\in\End\big(\R^{p+q}\big)$ as $X=\smtrx{A&B\\C&D}$.The Lie algebra of $\sSO(p,q)$ is given by
\begin{gather*}\fso(p,q)=\left\{\mtrx{A&B\\C&D}\Big|\mtrx{A&B\\C&D}^T\mtrx{\Id&0\\0&-\Id}+\mtrx{\Id&0\\0&-\Id}\mtrx{A&B\\C&D}=0\right\}.\end{gather*}
This implies that $A\in\fso(p)$, $D\in\fso(q)$ and $B=-C^T$, thus the Cartan decomposition is given by
\begin{gather*}\fso(p,q)=(\fso(p)\oplus\fso(q))\oplus \Hom\big(\R^p,\R^q\big).\end{gather*}

Similar to the previous examples, we use the standard representation. Since $\sH^\C=\sS(\sO(p,\C)\times\sO(q,\C))$, the restriction of the standard representation of $\sSO(p+q,\C)$ preserves an orthogonal splitting $\C^{p+q}=\C^p\oplus\C^q$. As in the previous example, a holomorphic principal $\sSO(p+q,\C)$-bundle is equivalent to a triple $(E,\omega,Q_E)$. A holomorphic principal $\sS(\sO(p,\C)\times\sO(q,\C))$-bundle is thus equivalent to a triple $(E,\omega,Q_E)$ which decomposes as
\begin{gather*}(E,\omega,Q_E)=\left(V\oplus W, \omega,\smtrx{Q_V&\\&-Q_W}\right),\end{gather*}
where $V$ and $W$ respectively have rank $p$ and $q$ and quadratic forms $Q_V$ and $Q_W$. Using the Cartan decomposition and the description of the Lie algebra, the Higgs field $\Phi\in H^0(\End(V\oplus W)\otimes K)$ is given by
\begin{gather*}\Phi=\mtrx{0&\eta^\dagger\\\eta&0},\end{gather*}
where $\eta\in H^0(\Hom(V,W)\otimes K)$ and, regarding $Q_V$ and $Q_W$ as isomorphisms $V\to V^*$ and $W\to W^*$ respectively, $\eta^\dagger=-Q_V^{-1}\eta^TQ_W$.

\begin{Remark} Taking the determinant of the isomorphisms $Q_V\colon V\to V^*$ defines an isomorphism of determinant line bundles $\det(V)\cong \det(V^*)$, or equivalently $\det(V)^2\cong\Oo$. Thus, the determinant line bundle of an orthogonal bundle $(V,Q_V)$ on $X$ is one of the $2^{2g}$ order two points in the Jacobian of~$X$. The above volume form $\omega$ defines an isomorphism $\Lambda^{p+q}(V\oplus W)=\Lambda^pV\otimes\Lambda^qW\to\Oo$. Using the orthogonal structures, this implies that $\Lambda^pV\cong\Lambda^q W^*\cong\Lambda^q W$.
\end{Remark}

Since $\sSO(p,q)$-Higgs bundles will be a main object of study we record this in a proposition.
\begin{Proposition}\label{prop SO(p,q)Higgs bundle}
An $\sSO(p,q)$-Higgs bundle on $X$ is equivalent to the data
\begin{itemize}\itemsep=0pt
\item a holomorphic rank $p$ vector bundle $V\to X$,
\item a holomorphic symmetric isomorphism $Q_V\colon V\to V^*$,
\item a holomorphic rank $q$ vector bundle $W\to X$,
\item a holomorphic symmetric isomorphism $Q_W\colon W\to W^*$,
\item a holomorphic isomorphism $\omega\colon \Lambda^pV\to \Lambda^qW^*$,
\item a holomorphic section $\eta\in H^0(\Hom(V,W)\otimes K)$.
\end{itemize}
The $\sSO(p+q,\C)$-Higgs bundle associated to a tuple $(V,Q_V,W,Q_W,\omega,\eta)$ is
\begin{gather*}(E,\omega,Q_E,\Phi)=\left(V\oplus W,~\omega, \mtrx{Q_V&0\\0&-Q_W},\mtrx{0&\eta^\dagger\\\eta&0}\right),\end{gather*}
and the associated $\sSL(p+q,\C)$-Higgs bundle is given by forgetting $Q_E$.
\end{Proposition}
We will often suppress $Q_V$, $Q_W$ and $\omega$ from the notation and just refer to an $\sSO(p,q)$-Higgs bundle as a triple $(V,W,\eta)$. We will also denote the associated Higgs bundle schematically by
\begin{gather}
\label{eq quiver notation}\xymatrix{V\ar@/_.5pc/[r]_{\eta}&W,\ar@/_.5pc/[l]_{\eta^\dagger}}
\end{gather}
where we have suppressed the twisting by $K$ from the notation.
\begin{Remark}
Recall that the group $\sSO(p,q)$ has two connected components, $\sSO_0(p,q)<\sSO(p,q)$ denotes the connected component of the identity. The maximal compact subgroup of $\sSO_0(p,q)$ is $\sSO(p)\times \sSO(q)$. Thus, an $\sSO(p,q)$-Higgs bundle $(V,W,\eta)$ reduces to an $\sSO_0(p,q)$-Higgs bundle if and only if both $V$ and $W$ have trivial determinant.
\end{Remark}
\end{Example}
\subsection{Stability and the moduli space}
The moduli space of Higgs bundle parameterizes isomorphism classes of Higgs bundles. The isomorphism group for Higgs bundles is called the gauge group. Just as with the character variety, to get a nice moduli we restrict to a special class of Higgs bundles whose gauge orbits are closed.

Given a smooth principal $\sH^\C$-bundle $P\to X$, the $\sH^\C$-gauge group $\Gg_{\sH^\C}$ is the group of bundle automorphisms $f\colon P\to P$. The elements of $\Gg_{\sH^\C}$ are given by sections of an associated bundle of groups $P[\sH^\C]=P\times_{\Ad_{\sH^\C}}\sH^\C$:
\begin{gather*}\Gg_{\sH^\C}= \Omega^0\big(X,P\big[\sH^\C\big]\big).\end{gather*}

Recall that a holomorphic structure on a vector bundle $E$ is equivalent to a Dolbeault operator. That is, a differential ope\-ra\-tor
\begin{gather*}\bar\p_E\colon \ \Omega^0(E)\to\Omega^{0,1}(E)\end{gather*}
so that\footnote{Dolbeault operators must also satisfy the integrability condition $\bar\p_E^2=0$, but this is automatic on a Riemann surface.} $\bar\p_E(fs)=\bar\p f\otimes s+f\bar\p_Es$ for all functions $f\in\Omega^0(\C)$ and sections $s\in\Omega^0(E)$. Note that the $(0,1)$-part of a connection on $E$ defines a Dolbeault operator. In particular, the space of holomorphic structures on $E$ is an infinite-dimensional affine space with underlying vector space $\Omega^{0,1}(\End(E))$.

For principal bundles, an analogous theory holds. Namely a holomorphic structure on a~principal $\sH^\C$-bundle $P\to X$ is equivalent to a section $\bar\p_P\in\Omega^{0,1}\big(P,\fh^\C\big)$ which defines the $(0,1)$-part of a connection. In particular, a holomorphic structure on $P$ defines a~Dolbeault operator on any associated vector bundle. The space of holomorphic structures on $P$ is an infinite-dimensional affine space with the space of basic $\fh^\C$-valued $(0,1)$-forms as underlying vector space. Equivalently, this vector space is given by sections $\Omega^{0,1}\big(X,P\big[\fh^\C\big]\big)$ of the adjoint bundle of $P$.

If we fix a smooth $\sH^\C$-bundle $P\to X$, the set of all Higgs bundles with underlying bundle $P$ is given by
\begin{gather*}\Hh(\sG)=\big\{\big(\bar\p_P,\varphi\big)\,|\,\bar\p_P\varphi=0\big\}.\end{gather*}
Fixing a holomorphic structure on $P$ defines $\Hh(\sG)$ as a \emph{quadratic} subspace of a vector space:
\begin{gather*}\Hh(\sG)\hookrightarrow \Omega^{0,1}\big(P\big[\fh^\C\big]\big)\oplus \Omega^{1,0}\big(P\big[\fm^\C\big]\big).\end{gather*}
\begin{Remark} Note that when $\sG$ is complex $\Omega^{0,1}\big(P\big[\fh^\C\big]\big)\oplus \Omega^{1,0}\big(P\big[\fm^\C\big]\big)\cong\Omega^1(P[\fg])$ since $\fh^\C\oplus \fm^\C\cong\fg\oplus\fg$.
\end{Remark}
For $(\alpha,\psi)\in \Omega^{0,1}\big(P\big[\fh^\C\big]\big)\oplus \Omega^{1,0}\big(P\big[\fm^\C\big]\big)$, we have $\big(\bar\p_P+\alpha,\varphi+\psi\big)\in\Hh(\sG)$ if
\begin{gather*}\bar\p_P\varphi+\bar\p_P\psi+[\alpha,\varphi]+[\alpha,\psi]=0.\end{gather*}
The tangent space is thus given by sections $(\alpha,\psi)$ satisfying this equation to first order:
\begin{gather*}
T_{\bar\p_P,\varphi}\Hh(\sG)=\big\{(\alpha,\psi)\in \Omega^{0,1}\big(P\big[\fh^\C\big]\big)\oplus \Omega^{1,0}\big(P\big[\fm^\C\big]\big)\,|\,\bar\p_P\psi+[\alpha,\varphi]=0\big\}.
\end{gather*}

\begin{Remark}\label{remark complex and symplectic forms}The space of Higgs bundles $\Hh(\sG)$ has a natural complex structure given by
\begin{gather*}I(\alpha,\psi)=({\rm i}\alpha,{\rm i}\psi).\end{gather*}
When $\sG$ is complex, $\Hh(\sG)$ also has a natural complex symplectic form given by
\begin{gather*}\Omega_I^\C((\alpha_1,\psi_1),(\alpha_2,\psi_2))={\rm i}\int_X\tr(\psi_2\wedge\alpha_1-\psi_1\wedge \alpha_2).\end{gather*}
For real groups $\sG$, it can be shown that $\Hh(\sG)\subset\Hh\big(\sG^\C\big)$ is a Lagrangian subspace.
\end{Remark}

The gauge group $\Gg_{\sH^\C}$ acts on $\Hh(\sG)$ by pullback, namely for $g\in\Gg_{\sH^\C}$
\begin{gather*}\big(\bar\p_P,\varphi\big)\cdot g=\big(\Ad_{g^{-1}}\bar\p_P,\Ad_{g^{-1}}\varphi\big).\end{gather*}
The orbits of the gauge group are not closed, and, to form a nice moduli space, we need a notion of (poly)stability. The moduli space $\Mm(\sG)$ of $\sG$-Higgs bundles is then defined to be the set $\Gg_{\sH^\C}$-orbits of polystable $\sG$-Higgs bundles. The orbits of the $\Gg_{\sH^\C}$-action on $\Hh(\sG)^{\rm ps}$ are closed (in an appropriate space) and the moduli space $\mathcal{M}(\sG)$ becomes a Hausdorff topological space.\footnote{For technical reasons, one needs to work with suitable Sobolev completions to give the moduli space a topology; see \cite{AtiyahBottYangMillsEq}, and also \cite[Section~8]{HauselThaddeusMirror} where the straightforward adaptation to Higgs bundles is discussed in the case $\sG=\sGL(n,\C)$.}
\begin{Remark}\label{Rem: symplectic form} Using the harmonic metric from Remark~\ref{herm metric remark}, one can also define a symplectic structure $\omega_I$ on $\Mm(\sG)$ which is a K\"ahler form for the complex structure $I$. In fact, when $\sG$ is complex, the moduli space $\Mm(\sG)$ is hyper-K\"ahler. We will not focus on this structure in this article.
\end{Remark}

In general, the notion of stability involves considering how all holomorphic structure group reductions of an $\sH^\C$-bundle to a parabolic subgroup interact with the Higgs field (see \cite{HiggsPairsSTABILITY}).
Instead of developing this theory in general, we will develop the appropriate stability conditions in the vector bundle situation.

Recall that a $\sGL(n,\C)$-Higgs bundle is equivalent to a rank $n$ holomorphic vector bundle $E$ and a section $\Phi\in H^0(\End(E)\otimes K)$. For $\sSL(n,\C)$ the bundle $E$ is equipped with a trivialization of $\Lambda^nE$, thus, $\deg(E)=0$.
\begin{Definition}\label{def Sln stab}
An $\sSL(n,\C)$-Higgs bundle $(E,\Phi)$ is
\begin{itemize}\itemsep=0pt
 \item semistable if for all proper holomorphic subbundles $F\subset E$ such that $\Phi(F)\subset F\otimes K$ we have $\deg(F)\leq0$,
 \item stable if for all proper holomorphic subbundles $F\subset E$ such that $\Phi(F)\subset F\otimes K$ we have $\deg(F)<0$, and
 \item polystable if $(E,\Phi)=\bigoplus_j(E_j,\Phi_j)$ with $(E_j,\Phi_j)$ stable and $\deg(E_j)=0$ for all $j$.
 \end{itemize}
\end{Definition}
For general groups $\sG$ the notion of semistability and polystability is functorial in the sense that if $\sG$ is a real form of a reductive subgroup of $\sSL(n,\C)$, then a $\sG$-Higgs bundle is semistable (respectively polystable) if and only if the associated $\sSL(n,\C)$-Higgs bundle is semistable (respectively polystable).
Moreover, the set of semistable $\sG$-Higgs bundles is open in $\Hh(\sG)$.

Let $\Hh^{\rm ps}(\sG)\subset\Hh(\sG)$ denote the set of polystable Higgs bundles. The gauge group $\Gg_{\sH^\C}$-preserves $\Hh^{\rm ps}(\sG)$, and the gauge orbits in $\Hh^{\rm ps}(\sG)$ are closed in the open set of semistable $\sG$-Higgs bundles. We define the moduli space $\Mm(\sG)$ to be the quotient space
\begin{gather*}\Mm(\sG)=\Hh^{\rm ps}(\sG)/\Gg_{\sH^\C}.\end{gather*}
We note that the complex structure $I$ (and the complex symplectic form $\Omega_I^\C$ when $\sG$ is complex) from Remark~\ref{remark complex and symplectic forms} are preserved by the gauge group action and thus descend to the moduli space.

For the general notion of stability, it is not the case that a $\sG$-Higgs bundle is stable if and only if the associated $\sSL(n,\C)$-Higgs bundle is stable. However, one can detect stable $\sG$-Higgs bundles inside of the set of polystable Higgs bundles with the following proposition.
\begin{Proposition}\label{Prop stability for GC ss}Let $\sG$ be a real form of a complex semisimple subgroup of $\sSL(n,\C)$. A $\sG$-Higgs bundle $(\Pp,\varphi)$ is stable if it is polystable and has finite automorphism group. Moreover, the set of stable $\sG$-Higgs bundles is open in $\Hh(\sG)$.
\end{Proposition}
\begin{Remark}Note that if $(\Pp,\varphi)$ is a $\sG$-Higgs bundle whose associated $\sSL(n,\C)$-Higgs bundle is stable as an $\sSL(n,\C)$-Higgs bundle, then $(\Pp,\varphi)$ is stable as a $\sG$-Higgs bundle.
\end{Remark}
 Let $\Hh^{s}(\sG)\subset\Hh^{\rm ps}(\sG)$ be the stable locus, the quotient
 \begin{gather*}\Hh^{s}/\Gg_{\sH^\C}\subset\Mm(\sG)\end{gather*}
 is an orbifold. At a stable Higgs bundle one can show that the real dimension of the tangent space to $T_{[\bar\p_P,\varphi]}\Mm(\sG)$ is $\dim_\R(\sG)(2g-2)$ (see Remark \ref{rem dim of tangent space}). Thus the real dimension of $\Mm(\sG)$ is given by $\dim_\R(\sG)(2g-2)$.

\section[$\sSO(1,q)$-Higgs bundles especially when $q=2$]{$\boldsymbol{\sSO(1,q)}$-Higgs bundles especially when $\boldsymbol{q=2}$}\label{section SO1q and Hitchin comp}

In this section we will describe the moduli space of $\sSO(1,q)$-Higgs bundles and $\sSO_0(1,q)$-Higgs bundles. When $q=2$ we have $\sSO_0(1,2)\cong\sP\sSL(2,\R)$. In this case we will recall Hitchin's parameterization of all but one of the components of $\Mm(\sSO_0(1,2))$. In particular, we recall the Higgs bundle parameterization of Teichm\"uller space.

Recall from Proposition \ref{prop SO(p,q)Higgs bundle} that an $\sSO(1,n)$-Higgs bundle consists of a tuple $(V,Q_V,W,Q_W$, $\omega,\eta)$, where $\rk(V)=1$ and $\rk(W)=q$.
We can take $(V,Q_V)=(\Lambda^nW,\det(Q_W))$ and $\omega=\det(Q_W)\colon V\to\Lambda^qW^*$. Thus, such a tuple is determined by the triple $(W,Q_W,\eta)$, where
\begin{gather*}\eta\in H^0\big(W\otimes (\Lambda^qW)^{-1}\otimes K\big).\end{gather*}
Using the notation from \eqref{eq quiver notation}, the associated $\sSL(1+q,\C)$-Higgs bundle is given by
\begin{gather*}\xymatrix{\Lambda^nW\ar@/_.5pc/[r]_{\eta}&W.\ar@/_.5pc/[l]_{\eta^\dagger}}\end{gather*}

When $q=1$, we have $\eta\in H^0(K)$ and the first Stiefel--Whitney class $sw_1(W)\in H^1(X,\Z_2)$ of $W$ labels the components of $\Mm(\sSO(1,1))$. Namely,
\begin{gather*}\Mm(\sSO(1,1))=\coprod\limits_{sw_1\in H^1(X,\Z_2)}\Mm_{sw_1}(\sSO(1,1)),\end{gather*}
and each space $\Mm_{sw_1}(\sSO(1,1))$ is parameterized by $H^0(K)$.

For $q>1$, the first and second Stiefel--Whitney classes $(sw_1,sw_2)\in H^1(X,\Z_2)\times H^2(X,\Z_2)$ of $(W,Q_W)$ give a decomposition of the moduli space
\begin{gather*}\Mm(\sSO(1,n))=\coprod\limits_{sw_1,sw_2}\Mm_{sw_1}^{sw_2}(\sSO(1,n)).\end{gather*}

The first Stiefel--Whitney class of $W$ vanishes if and only if the $\sO(q,\C)$-bundle reduces to $\sSO(q,\C)$. Thus, for $q=2$ and $sw_1=0$, the bundle $W$ reduces to an $\sSO(2,\C)$-bundle. Since $\C^*\cong \sSO(2,\C)$, in this case the degree of the $\C^*$-bundle provides a refinement of the second Stiefel--Whitney class. More precisely, if $sw_1(W,Q_W)=0$, then there is a line bundle $L\in\Pic(X)$ such that
\begin{gather*}(W,Q_W)\cong\left(L\oplus L^{-1},\mtrx{0&1\\1&0}\right).\end{gather*}
The integer $\deg(L)$ satisfies $sw_2(W,Q_W)=\deg(L)\mod 2$, and the isomorphism switching $L$ with $L^{-1}$ preserves the $\sO(2,\C)$-structure. Thus, $|\deg(L)|\in \N$ which is a well defined invariant of $\sO(2,\C)$-bundles with vanishing~$sw_1$.

This gives a decomposition of the moduli space as
\begin{gather*}\coprod\limits_{sw_1\neq0,sw_2}\Mm^{sw_2}_{sw_1}(\sSO(1,2))~\amalg~ \coprod\limits_{d\in\N}\Mm_d(\sSO(1,2)).\end{gather*}
For Higgs bundles in $\Mm_d(\sSO(1,2))$ the splitting $W=L\oplus L^{-1}$ gives a decomposition of the Higgs field $\eta\colon \Oo\to W\otimes K$ as
\begin{gather*}\eta=\mtrx{\beta\\\gamma}\colon \ \Oo\to LK\oplus L^{-1}K,\end{gather*}
where $\beta\in H^0(LK)$ and $\gamma\in H^0\big(L^{-1}K\big)$.
Using $Q_W=\smtrx{0&1\\1&0}$, we can write the associated $\sSL(3,\C)$-Higgs bundle schematically as
\begin{gather}
\label{eq sl3 schematic}\xymatrix{L\ar@/_1pc/[r]_\gamma&\Oo\ar@/_1pc/[r]_\gamma\ar@/_1pc/[l]_\beta&L^{-1},\ar@/_1pc/[l]_\beta}
\end{gather}
where we recall that we suppress the twisting by $K$ from the notation.

The stability condition limits the objects we are considering.

\begin{Proposition}If $\big(\Oo,L\oplus L^{-1},\smtrx{0&1\\1&0},\smtrx{\beta\\\gamma}\big)$ is a polystable $\sSO(1,2)$-Higgs bundle with vanishing first Stiefel--Whitney class, then $|\deg(L)|\leq 2g-2$. Moreover, if $\deg(L)\in(0,2g-2]$, then $\gamma\neq0$ and if $\deg(L)\in[2-2g,0)$, then $\beta\neq0$.
\end{Proposition}
\begin{proof}Consider the associated $\sSL(3,\C)$-Higgs bundle \eqref{eq sl3 schematic}. By stability, if $\deg(L)>0$ then $\gamma\neq 0$ since otherwise $L$ would define a positive degree invariant subbundle. But, $\gamma\in H^0\big(L^{-1}K\big)$ so if $\deg(L)>2g-2$ then $\gamma=0$, contradicting stability. Similarly, if $\deg(L)<0$, then stability forces $\beta\neq0$ and we conclude $\deg(L)>2-2g$.
\end{proof}

For $d=|\deg(L)|>0$, we can parameterize the moduli space $\Mm_{d}(\sSO(1,2))$, this was done by Hitchin in \cite{selfduality} for the group $\sP\sSL(2,\R)$.
\begin{Theorem}[Hitchin \cite{selfduality}]\label{thm so12 comp} For $d>0$, the moduli space $\Mm_d(\sSO(1,2))$ is smooth and diffeomorphic to the total space of a rank $(d+g-1)$-complex vector bundle over the $(2g-2-d)$-symmetric product $\Sym^{2g-2-d}(X)$ of the Riemann surface~$X$.
\end{Theorem}
\begin{proof}By the above discussion, a point in $\Mm_d(\sSO(1,2))$ is determined by a triple $(L,\gamma,\beta)$ where $L\in\Pic^d(X)$, $\gamma\in H^0\big(L^{-1}K\big)\setminus\{0\}$ and $\beta\in H^0(LK)$.
The $\sS(\sO(1,\C)\times \sO(2,\C))$-bundle is given by
\begin{gather*}(V,Q_V,W,Q_W)=\left(\Oo,\mtrx{1},L\oplus L^{-1},\mtrx{0&1\\1&0}\right),\end{gather*}
and the Higgs field is $\eta=\smtrx{\beta\\\gamma}\colon V\to W\otimes K$.

For two triples $(L,\beta,\gamma)$ and $(L',\beta',\gamma')$ to define isomorphic $\sSO(1,2)$-Higgs bundles it is necessary that $|\deg(L)|=|\deg(L')|$. Thus we may assume $L=L'$ as elements $\Pic^d(X)$. The remaining holomorphic gauge transformation of the $\sS(\sO(1,\C)\times\sO(2,\C))$ bundle is given by
\begin{gather*}(g_V,g_W)=\left(1,\mtrx{\lambda&0\\0&\lambda^{-1}}\right),\end{gather*}
for $\lambda\in\C^*$. This gauge transformation acts on the Higgs field by
\begin{gather*}g_W^{-1}\eta g_V=\mtrx{\lambda^{-1}&0\\0&\lambda}\mtrx{\beta\\\gamma}\mtrx{1}=\mtrx{\lambda^{-1}\beta\\\lambda\gamma}.\end{gather*}
In particular, we note that the automorphism group of such an $\sSO(1,2)$-Higgs bundle is trivial since $\gamma\neq0$. Thus, the moduli space $\Mm_d(\sSO(1,2))$ is smooth and given by $\C^*$-equivalence classes $[L,\beta,\gamma]$ where $(L,\beta,\gamma)\sim (L',\beta',\gamma')$ if and only if $L=L'\in\Pic^d(X)$, $\beta=\lambda\beta'$ and $\gamma=\lambda^{-1}\gamma'$ for $\lambda\in\C^*$.

Recall that the space of effective divisors on $X$ of degree $n$ is given by the $n^{\rm th}$-symmetric product $\Sym^n(X)$. Taking the projective class of $\gamma\in H^0\big(L^{-1}K\big)\setminus\{0\}$ defines a surjective map to the space of effective degree $2g-2-d$ divisors on $X$:
\begin{gather*}\xymatrix@R=0em{\Mm_d(\sSO(1,2))\ar[r]&\Sym^{2g-2-d}(X),\\[L,\gamma,\beta]\ar@{|->}[r]&[\gamma].}\end{gather*}
We claim that the fiber of this map is a vector space of rank $(d+g-1)$. Denote by $\Oo([\gamma])$ the line bundle associated to the divisor~$[\gamma]$. The line bundle $L$ is given by $L=\Oo([\gamma])^{-1}K$ and $\beta\in H^0\big(\Oo([\gamma])^{-1}K^2\big)$. Thus $L$ is determined by $[\gamma]$ and $\beta$ can be any element of the $(d+g-1)$-dimensional vector space $H^0\big(\Oo([\gamma])^{-1}K^2\big)$.
\end{proof}

We now collect many corollaries of the above theorem.
\begin{Corollary}For $d>0$ the moduli space$\Mm_d(\sSO(1,2))$ is connected and homotopy equivalent to the symmetric product $\Sym^{2g-2-d}(X)$.
\end{Corollary}
The cohomology ring of a symmetric product of a Riemann surface was computed in \cite{SymmetricProductsofAlgebraicCurves}, as a result this computes the cohomology ring of $\Mm_d(\sSO(1,2))$. When $d=2g-2$, the space is contractible.

Consider the following map
\begin{gather*}\xymatrix@R=0em{\Mm(\sSO(3,\C))\ar[r]& H^0\big(K^2\big),\\[E,Q_E,\Phi]\ar@{|->}[r]&\frac{1}{4}\tr\big(\Phi^2\big).}\end{gather*}
This is the Hitchin fibration for $\sSO(3,\C)$, we will discuss the Hitchin fibration in more generality in subsequent sections.
\begin{Corollary}The moduli space $\Mm_{2g-2}(\sSO(1,2))$ is parameterized by the $(3g-3)$-dimensional complex vector space $H^0\big(K^2\big)$ of holomorphic differentials. Moreover, $\Mm_{2g-2}(\sSO(1,2))$ is the image of a section of the $\sSO(3,\C)$-Hitchin fibration.
\end{Corollary}
\begin{proof}
An $\sSO(1,2)$-Higgs bundle in $\Mm_{2g-2}(\sSO(1,2))$ is determined by a triple $(L,\beta,\gamma)$ where $\deg(L)=2g-2$, $\beta\in H^0( LK)$ and $\gamma\in H^0\big(L^{-1}K\big)\setminus\{0\}$. The condition on $\gamma$ implies that $L=K$ and thus $\beta\in H^0\big(K^2\big)$. If we normalize $\gamma$ to by $\gamma=1\in H^0(\Oo)$, then there is no more gauge freedom, and so
\begin{gather*}\Mm_{2g-2}(\sSO(1,2))\cong H^0\big(K^2\big).\end{gather*}

Using the above parameterization of $\Mm_{2g-2}(\sSO(1,2))$ by $H^0\big(K^2\big)$, the $\sSO(3,\C)$-Higgs bundle associated to $q_2\in H^0\big(K^2\big)$ is given by
\begin{gather*}[E,Q_E,\Phi]=\left(\Oo \oplus K\oplus K^{-1} , \mtrx{-1&0&0\\0&0&1\\0&1&0},\mtrx{0&1&q_2\\q_2&0&0\\1&0&0}\right).\end{gather*}
For this Higgs bundle $\frac{1}{4}\tr\big(\Phi^2\big)=q_2$.
\end{proof}

Translating these statements to the character variety $\Xx(\Gamma,\sSO(1,2))$ via the nonabelian Hodge correspondence gives the following.
\begin{Corollary}
For each $0<d\leq 2g-2$, the character variety $\Xx(\Gamma,\sSO(1,2))$ has a connected component $\Xx_d(\sSO(1,2))$ which is smooth and diffeomorphic to a real rank $2d+2g-2$ vector bundle over the symmetric product $\Sym^{2g-2-d}(S)$.
\end{Corollary}
\begin{Corollary}Every representation $\rho\in \Xx_d(\Gamma,\sSO(1,2))$ factors through the connected component of the identity $\sSO_0(1,2)$ and the Fuchsian representations are given by
\begin{gather*}\Xx_{2g-2}(\Gamma,\sSO(1,2))\cong\Fuch(\Gamma).\end{gather*}
\end{Corollary}
\begin{proof}The space $\Fuch(\Gamma)$ consists of two connected component of the character variety \linebreak $\Xx(\Gamma,\sP\sSL(2,\R))=\Xx(\Gamma,\sSO_0(1,2))$ which are identified with the Teichm\"uller space of $S$. Since the representations in these components are conjugate by an element of $\sSO(1,2)$ which is not in $\sSO_0(1,2)$, the two components of $\Fuch(\Gamma)$ are identified in $\Xx(\Gamma, \sSO(1,2))$. Since $\Fuch(\Gamma)$ is contractible and the only~$d$ for which $\Xx_{d}(\Gamma,\sSO(1,2))$ is contractible is $d=2g-2$ we are done.
\end{proof}

\begin{Remark}\label{remark SL2 version}Since the second Stiefel--Whitney class invariant of the Higgs bundles in \linebreak $\Mm_d(\sSO(1,2))$ is given by $d\mod{2}$, the associated $\sSO(3,\C)$-Higgs bundles lift to $\sSpin(3,\C)$ if and only if~$d$ is even. Recall that the isomorphism $\sSpin(3,\C)=\sSL(2,\C)$, is given by the~$2$ to~$1$ map $\sSL(2,\C)\to \sSO(3,\C)$ which is induced by the action on the second symmetric pro\-duct~$S^2\big(\C^2\big)$. Here, the volume form on~$\C^2$ induces a nondegenerate symmetric form on the second symmetric product.

The $\sSL(2,\R)$-Higgs bundles which give rise to the Higgs bundles in $\Mm_{2d}(\sSO(1,2))$ are thus given by
\begin{gather*}(E,\omega,Q_E,\Phi)\cong\left(N\oplus N^{-1},\mtrx{0&-1\\1&0},\mtrx{0&1\\1&0},\mtrx{0&\beta\\\gamma&0}\right),\end{gather*}
where we view the volume form $\omega\in H^0\big(\Lambda^2E\big)$ is viewed as a skew symmetric homomorphism $\smtrx{0&-1\\1&0}\colon E^*\to E$.
Indeed, taking the second symmetric product gives
\begin{gather*}\big(S^2E,S^2\omega,S^2\Phi\big)\cong\left( \Oo\oplus N^2\oplus N^{-2}, \mtrx{-1&0&0\\0&0&1\\0&1&0}, \mtrx{0&\gamma&\beta\\\beta&0&0\\\gamma&0&0}\right).\end{gather*}
In particular, the $\sSL(2,\R)$-Higgs bundles which define points in the component $\Mm_{2g-2}(\sSO(1,2))$ are given by
\begin{gather*}
\left(K^\frac{1}{2}\oplus K^{-\frac{1}{2}},\mtrx{0&-1\\1&0},\mtrx{0&1\\1&0},\mtrx{0&q_2\\1&0}\right)
\end{gather*}
for one of the $2^{2g}$ choices of square root $K^\frac{1}{2}$ of $K$. In particular, there are $2^{2g}$-connected components of $\Mm(\sSL(2,\R))$ which project to $\Mm_{2g-2}(\sP\sSL(2,\R))\cong\Mm_{2g-2}(\sSO(1,2))$.
\end{Remark}

\section{The Hitchin fibration and Hitchin section}\label{section Hitchin component}
So far we have seen that the character variety $\Xx(\Gamma,\sG)$ is homeomorphic to the moduli space of $\sG$-Higgs bundles. The upshot of this correspondence is that the Higgs bundle moduli space has a lot of useful structures which the character variety is lacking. In this section we define the Hitchin component and use this additional structure to construct the Hitchin component from Definition~\ref{Def Hit comp}.

\subsection{The Hitchin fibration}
Suppose $\sG$ is a complex simple Lie group. Similar to Chern--Weil theory, we can apply an invariant polynomial to the Higgs field and obtain a holomorphic differential. Fixing a homogeneous basis $p_1,\dots,p_{\rk(\sG)}$ of the $\Ad_\sG$-invariant polynomials $\C[\fg]$ with $\deg(p_j)=m_j+1$ defines a map
\begin{gather}\label{eq hitchin fibration}\begin{split}&
h\colon \ \xymatrix@R=0em{\Mm(\sG)\ar[r]&\displaystyle\bigoplus\limits_{j=1}^{\rk(\sG)}H^0\big(K^{m_j+1}\big),\\\big[\bar\p_P,\varphi\big]\ar@{|->}[r]&(p_1(\varphi),\dots,p_{\rk(\sG)}(\varphi))}\end{split}
\end{gather}
called the Hitchin fibration. For example, when $\sG=\sSL(n,\C)$ we have $m_j=j$ for $1\leq j\leq n-1$, and when $\sG=\sSO(2n+1,\C)$ we have $m_j=2j-1$ for $1\leq j\leq n$.

In general, a computation using the Riemann--Roch theorem shows that the base is half the dimension of the moduli space:
\begin{gather*}\dim_\C\left(\bigoplus\limits_{j=1}^{\rk(\sG)}H^0\big(K^{m_j+1}\big)\right)=\frac{1}{2}\dim_\C\big(\Mm(\sG)\big)=\dim_\C(\sG)(g-1).\end{gather*}
Moreover, the Hitchin fibration is a proper map \cite{NitsurePairs}. In fact, the generic fibers of the Hitchin fibration are half-dimensional tori and this makes $\Mm(\sG)$ into a algebraic completely integrable system \cite{IntSystemFibration}, we will not make use of this additional structure.

\begin{Remark}\label{remark dim of hitchin base}Notice that the dimension of the base of the Hitchin fibration is the same as the dimension of the moduli space of $\sG^r$-Higgs bundles for $\sG^r<\sG$ any real form. For example, the Hitchin base of $\sSO(2n+1,\C)$ has the same dimension as $\Mm(\sSO(p,q))$ for all $p$ and $q$ satisfying $p+q=2n+1$.
\end{Remark}

\subsection{The Hitchin section}
Let $\fg$ be a semisimple complex Lie algebra. For $\fs\subset\fg$ a subalgebra isomorphic to $\fsl(2,\C)$, consider the decomposition of $\fg$ into irreducible $\fsl(2,\C)$-representations
\begin{gather*}\fg=\bigoplus\limits_{j=1}^N V_j.\end{gather*}
For any such $\fs\subset\fg$ we have $N\geq\rk(\fg)$, and when $N=\rk(\fg)$ the three-dimensional subalgebra $\fs$ is called \emph{principal}. Up to conjugation, there is a unique principal three-dimensional subalgebra~\cite{ptds}. In this case we have $\dim(V_j)=2m_j+1$ where $1=m_1\leq m_2\leq \cdots\leq m_{\rk(\fg)}$ are the exponents of~$\fg$. Moreover, when we restrict a principal embedding~$\fsl(2,\C)\to \fg$ to the real subalgebra~$\fsl(2,\R)$, the image lies in a split real subalgebra of~$\fg$. This defines an embedding
\begin{gather*}\iota_{\rm pr}\colon \ \sP\sSL(2,\R)\to \sG^{\rm split}.\end{gather*}

\begin{Theorem}[Hitchin \cite{liegroupsteichmuller}] Let $\sG$ be a complex simple Lie group, then the Hitchin fibration~\eqref{eq hitchin fibration} has a section
\begin{gather*}s_h\colon \ \bigoplus\limits_{j=1}^{\rk(\sG)}H^0\big(K^{m_j+1}\big)\longrightarrow \Mm(\sG),\end{gather*}
which maps onto a component of the moduli space for the split real form $\Mm\big(\sG^{\rm split}\big)$. Under the nonabelian Hodge correspondence $($Theorem~{\rm \ref{Thm: Nonabelian Hodge Correspondence})}, the image of this section defines the Hitchin component $\Hit\big(\sG^{\rm split}\big)\subset\Xx\big(\sG^{\rm split}\big)$ from Definition~{\rm \ref{Def Hit comp}}.
\end{Theorem}

We will prove the above theorem for $\sG=\sSO(2p+1,\C)$, namely we will construct the Hitchin section and prove that it maps onto a component for the group $\sSO(p,p+1)$. For $\sSO(2p+1,\C)$ the Hitchin fibration is given by
\begin{gather*}\Mm(\sSO(2p+1,\C))\to \bigoplus\limits_{j=1}^{2p}H^0\big(K^{2j}\big).\end{gather*}
Consider the rank $p$ holomorphic orthogonal bundle
\begin{gather}
\label{eq Kk_p notation}\Kk_p= K^{p-1}\oplus K^{p-3}\oplus\cdots \oplus K^{3-p}\oplus K^{1-p}.
\end{gather}
Note that $\Kk_p$ has a natural orthogonal structure $Q_p=\smtrx{&&1\\&\iddots&\\1}\colon \Kk_p\to\Kk_p^*$.

Consider the map
\begin{gather*}\widehat\Psi\colon \ \xymatrix{\bigoplus\limits_{j=1}^{2p}H^0\big(K^{2j}\big)\ar[r]&\Hh(\sSO(p,p+1))}\end{gather*}
defined by
\begin{gather}
\label{eq image Psi hat}\widehat\Psi(q_2,\dots,q_{2p})=\left(\Kk_p,Q_p,\Kk_{p+1},Q_{p+1},\mtrx{q_2&q_4&\dots&q_{2p}\\1&q_2&\dots&q_{2p-2}\\&\ddots&\ddots&\vdots\\&&1&q_2\\&&&1}\colon \Kk_{p}\to \Kk_{p+1}\otimes K\right).
\end{gather}
We claim that the image of $\widehat\Psi$ is contained in the stable Higgs bundles $\Hh^{s}(\sSO(p,p+1))$ and that the induced map $\Psi\colon \bigoplus\limits_{j=1}^{p}H^0\big(K^{2j}\big)\to\Mm(\sSO(p,p+1))$ is injective and has an open and closed image.

\begin{Proposition}The image of $\Psi$ consists of stable $\sSO(p,p+1)$-Higgs bundles.
\end{Proposition}
\begin{proof}
Consider the stable $\sSL(2,\C)$-Higgs bundle
\begin{gather}
\label{eq rnk 2}(E,\Phi)= \left(K^\frac{1}{2}\oplus K^{-\frac{1}{2}},\mtrx{0&0\\1&0}\right).
\end{gather}
Since the unique irreducible $(2p+1)$-dimensional representation of $\sSL(2,\C)$ is given by the $2p$-symmetric product, the $\sSL(2p+1,\C)$-Higgs bundle given by
\begin{gather*}\big(S^{2p+1}E,S^{2p+1}\Phi\big)= \left(\!K^{p}\oplus K^{p-1}\oplus\cdots\oplus K^{1-p}\oplus K^{-p},\mtrx{0&\\p-1&0\\&p-2&0\\&&\ddots&\ddots\\&&&p-1&0}\!\right)\end{gather*}
is also stable. Moreover this is gauge equivalent to
\begin{gather}
\label{eq hit fixed point}\left(K^{p}\oplus K^{p-1}\oplus\cdots\oplus K^{1-p}\oplus K^{-p},\mtrx{0&\\1&0\\&\ddots&\ddots\\&&1&0}\right).
\end{gather}

After rearranging the summands of $\Kk_p\oplus\Kk_{p+1}$, the $\sSL(2p+1,\C)$-Higgs bundle associated to $\widehat\Psi(0,\dots,0)$ is given by~\eqref{eq hit fixed point}. Thus, $\widehat\Psi(0,\dots,0)$ is a stable $\sSO(p,p+1)$-Higgs bundle. Since sta\-bility is an open condition, for $q_2,\dots,q_{2p}$ sufficiently close to zero, the Higgs bundle $\widehat\Psi(q_2,\dots,q_{2p})$ is also stable.

Scaling the Higgs field does not preserve the image of $\widehat\Psi$. However, for each $\lambda\in\C^*$, the Higgs bundle obtained by scaling the Higgs field of $\widehat\Psi(q_2,q_4,\dots,q_{2p})$ by $\lambda$ is gauge equivalent to $\widehat\Psi\big(\lambda^2q_2,\lambda^{4}q_4,\dots,\lambda^{2p}q_{2p}\big)$. Since stability is preserved by scaling the Higgs field, all Higgs bundles in the image of $\widehat\Psi$ are stable.
\end{proof}

\begin{Proposition}\label{prop hom basis}Let $\Phi(q_2,\dots,q_{2p})$ be the Higgs field of the $\sSO(2p+1,\C)$-Higgs bundle associated to the $\sSO(p,p+1)$-Higgs bundle $\widehat\Psi(q_2,\dots,q_{2p})$. There is a basis $(p_1,\dots,p_p)$ of the invariant polynomials $\C[\fso(2p+1,\C)]^{\sSO(2p+1,\C)}$ so that for all $j$
\begin{gather*}p_j(\Phi(q_2,\dots,q_{2p}))=q_{2j}.\end{gather*}
\end{Proposition}
\begin{proof}In the general setting of a complex semisimple Lie group the existence of such a basis was proven by Kostant in~\cite{ptds}. For $\sSO(2p+1,\C)$ we construct such a basis by direct computation. We explain how this works for $p=2$ and leave the general case to the reader.

After rearranging the summands, the $\sSO(5,\C)$-Higgs bundle $(E,Q,\Phi)$ associated to the $\sSO(2,3)$-Higgs bundle $\widehat\Psi(q_2,q_4)$ is given by
\begin{gather*}\left(K^2\oplus K\oplus \Oo\oplus K^{-1}\oplus K^{-2},\mtrx{&&&&-1\\&&&1\\&&-1\\&1\\-1}, \mtrx{0&q_2&0&q_4&0\\1&0&q_2&0&q_4\\0&1&0&q_2&0\\0&0&1&0&q_2\\0&0&0&1&0}\right).\end{gather*}
We have $\tr\big(\Phi^2\big)=8q_2$ and $\tr\big(\Phi^4\big)=14q_2^2+5q_4$, thus we choose the basis
\begin{gather*}\xymatrix{p_1(\Phi)=\frac{1}{8}\tr\big(\Phi^2\big)&\text{and}&p_2(\Phi)=\frac{1}{5}\tr\big(\Phi^4\big)-\frac{1}{14}\tr\big(\Phi^2\big)^2}.\end{gather*}
\end{proof}

By the previous two propositions, the map $\widehat\Psi$ gives rise to a well defined map
\begin{gather*}\Psi\colon \ \bigoplus\limits_{j=1}^{p} H^0\big(K^{2j}\big)\to \Mm(\sSO(p,p+1)), \end{gather*}
which is a section of the Hitchin fibration for $\Mm(\sSO(2p+1,\C))\to\bigoplus\limits_{j=1}^p H^0\big(K^{2j}\big)$. We now show that the image of $\Psi$ is open and closed.

\begin{Proposition}The image of the map $\Psi\colon \bigoplus\limits_{j=1}^{p} H^0\big(K^{2j}\big)\to \Mm(\sSO(p,p+1))$ is open and closed.
\end{Proposition}
\begin{proof}For openness, we use invariance of domains. Namely, the spaces have the same dimension, and, by Proposition~\ref{prop hom basis}, no two Higgs bundles in the image of $\widehat\Psi$ are gauge equivalent. Thus, $\widehat\Psi$ is an injective map between manifolds of the same dimension, and is therefore open.

For closedness suppose $\big(q_2^j,\dots,q_{2p}^j\big)$ is a divergent sequence of points in $\bigoplus\limits_{j=1}^{p}H^0\big(K^{2j}\big)$. By Pro\-po\-si\-tion~\ref{prop hom basis} and properness of the Hitchin fibration we conclude that the sequence $\Psi\big(q_2^j,\dots,q_{2p}^j\big)$ also diverges in $\Mm(\sSO(p,p+1))$.
\end{proof}

To complete the proof we need to show that under the nonabelian Hodge correspondence, the~com\-po\-nent defined by $\Psi\Big(\bigoplus\limits_{j=1}^{p}H^0\big(K^{2j}\big)\Big)$ is the Hitchin component $\Hit(\sSO(p,p+1))$ from De\-fi\-ni\-tion~\ref{Def Hit comp}. It suffices to show that the representation associated to $\Psi(0,\dots,0)$ is in $\Hit(\sSO(p,p+1))$. By Remark~\ref{remark SL2 version}, the Higgs bundle~\eqref{eq rnk 2} defines an $\sSL(2,\R)$-Higgs bundle whose corresponding representation is in $\Fuch(\Gamma)$. From Example~\ref{ex SO(p,p+1) principal}, the principal embedding $\iota_{\rm pr}\colon \sP\sSL(2,\R)\to\sSO(p,p+1)$ is given by taking the $2p$-symmetric product of the standard representation of $\sSL(2,\R)$. Thus, the representation associated to $\Psi(0,\dots,0)$ is contained in $\Hit(\sSO(p,p+1))$.

\section{Structure of the moduli space}
\subsection{Tangent space and deformation complex}
In this section we will assume for simplicity that $\sG$ is a real form of a complex semisimple Lie group. Under this assumption, the automorphism group of a stable $\sG$-Higgs bundle is discrete (see Proposition \ref{Prop stability for GC ss}). Recall that $\Hh(\sG)$ is the set of pairs $\big(\bar\p_P,\varphi\big)$ where $\bar\p_P$ is a Dolbeault operator on a~smooth $\sH^\C$-bundle $P\to X$ and $\varphi\in \Omega^{1,0}\big(P\big[\fm^\C\big]\big)$ such that $\bar\p_P\varphi=0$.

Since the space of Dolbeault operators is an affine space with underlying vector space isomorphic $\Omega^{0,1}\big(P\big[\fh^\C\big]\big)$, the tangent space of $\Hh^{\rm ps}(\sG)$ at $\big(\bar\p_P,\varphi\big)$ is given by the set of $(\alpha,\psi)\in\Omega^{0,1}\big(P\big[\fh^\C\big]\big)\oplus\Omega^{1,0}\big(P\big[\fm^\C\big]\big)$ so that $\varphi+\psi$ is holomorphic with respect to the Dolbeault operator $(\bar\p_P+\alpha)$ to first order. That is,
\begin{gather*}T_{(\bar\p_P,\varphi)}\Hh(\sG)=\big\{(\alpha,\psi)\in\Omega^{0,1}\big(P\big[\fh^\C\big]\big)\oplus\Omega^{1,0}\big(P\big[\fm^\C\big]\big)\,|\,\bar\p_P\psi+[\alpha,\varphi] =0\in\Omega^{1,1}\big(P\big[\fh^\C\big]\big)\big\}.\end{gather*}

The moduli space of $\sG$-Higgs bundles is a set of gauge equivalence classes:
\begin{gather*}\Mm(\sG)=\Hh^{\rm ps}(\sG)/\Gg_{\sH^\C},\end{gather*}
where $\Hh^{\rm ps}(\sG)$ denotes the set of polystable pairs. At stable points of the moduli space, the~tangent space can be interpreted as a quotient of the tangent space to the gauge orbit $\Gg_{\sH^\C}\cdot\big(\bar\p_P,\varphi\big)$:
\begin{gather*}T_{[\bar\p_P,\varphi]}\Mm(\sG)= T_{(\bar\p_P,\varphi)}\Hh(\sG)/T_{(\bar\p_P,\varphi)}\Gg_{\sH^\C}\cdot\big(\bar\p_P,\varphi\big).\end{gather*}
This is because, under our assumption on $\sG$, the automorphism group of a stable $\sG$-Higgs bundle is discrete, and so the gauge group action is locally free. The tangent space to the gauge orbit of a stable $\sG$-Higgs bundle can thus be identified with the tangent space at the identity of the gauge group
\begin{gather*}T_e\Gg_{\sH^\C}\cong \Omega^0\big(P\big[\fh^\C\big]\big).\end{gather*}

The identification of $\Omega^0\big(P\big[\fh^\C\big]\big)$ with the tangent space $T_{(\bar\p_P,\varphi)}\Gg_{\sH^\C}\cdot\big(\bar\p_P,\varphi\big)$ is given by the map
\begin{gather*}\xymatrix@R=0em{\Omega^0\big(P\big[\fh^\C\big]\big)\ar[r] &\Omega^{0,1}\big(P\big[\fh^\C\big]\big)\oplus\Omega^{1,0}\big(P\big[\fm^\C\big]\big),\\x\ar@{|->}[r]&\big(\bar\p_Px,[\varphi,x]\big).}\end{gather*}
Note that $(\bar\p_Px,[\varphi,x])\in T_{(\bar\p_P,\varphi)}\Hh(\sG)$ since $\bar\p_P([\varphi,x])+[\bar\p_Px,\varphi]=0$.

The tangent space to $\Mm(\sG)$ at a stable Higgs bundle $\big[\bar\p_P,\varphi\big]$ is thus identified with
\begin{gather}\label{eq tangent space}
T_{[\bar\p_P,\varphi]}\Mm(\sG)\cong\dfrac{\big\{(\alpha,\psi)\in\Omega^{0,1}\big(P\big[\fh^\C\big]\big) \oplus\Omega^{1,0}\big(P\big[\fm^\C\big]\big)\,|\,\bar\p_P\psi+[\alpha,\varphi]=0)\big\}}{\big\{(\bar\p_Px,[\varphi,x])\in\Omega^{0,1} \big(P\big[\fh^\C\big]\big)\oplus\Omega^{1,0}\big(P\big[\fm^\C\big]\big)\,|\,x\in\Omega^0\big(P\big[\fh^\C\big]\big)\big\}}.
\end{gather}

 The tangent space fits into a very useful exact sequence.
\begin{Proposition}For a stable $\sG$-Higgs bundle $\big(\bar\p_P,\varphi\big)$, we have an exact sequence
\begin{gather}
\label{eq exact sequence}\begin{split} & \xymatrix@R=1em@C=2em{0\ar[r]&H^0\big(P\big[\fh^\C\big]\big)\ar[r]^{\ad_\varphi \ \ \ }&H^0\big(P\big[\fm^\C\big]\otimes K\big)\ar[r]^{\ \ i}&T_{[\bar\p_P,\varphi]}\Mm(\sG)\ar@{->}`r/3pt [d] `/10pt[l] `^d[lll] `^r/3pt[d][dll]_{\pi\ \ \ \ \ \ \ \ \ }\\&H^1\big(P\big[\fh^\C\big]\big)\ar[r]_{\ad_\varphi \ \ \ }&H^1\big(P\big[\fm^\C\big]\otimes K\big),&}
\end{split} \end{gather}
where the map $i$ is induced by the inclusion
\begin{gather*}\xymatrix@R=0em{H^0\big(P\big[\fm^\C\big]\otimes K\big)\ar[r]& T_{(\bar\p_P,\varphi)} \Hh(\sG),\\\psi\ar@{|->}[r]&(0,\psi),}\end{gather*} and the map $\pi$ is induced by the projection
\begin{gather*}\xymatrix@R=0em{ T_{(\bar\p_P,\varphi)} \Hh(\sG)\ar[r]&\Omega^{0,1}\big(P\big[\fh^\C\big]\big),\\(\psi,\alpha)\ar@{|->}[r]&\alpha.}\end{gather*}
\end{Proposition}
\begin{Remark}\label{rem dim of tangent space}In fact, the sequence \eqref{eq exact sequence} is exact on the right, however we have not developed the techniques to prove this. Using exactness of this sequence, a Riemann--Roch calculation implies that the real dimension of the tangent space at a stable Higgs bundle is given by $\dim_\R(\sG)(2g-2)$.
\end{Remark}
\begin{proof}
First, the elements in the kernel of the map $\ad_\varphi$ correspond to tangent vectors of one parameter families of automorphisms of $\big(\bar\p_P,\varphi\big)$. Thus stability implies that $\ker(\ad_\varphi)=0$.

Next, note that the kernel of the map $i\colon H^0\big(P\big[\fm^\C\big]\otimes K\big)\to T_{[\bar\p_P,\varphi]}\Mm(\sG)$ is given by the set of $(0,\psi)=\big(\bar\p_Px,[\varphi,x]\big)$. Thus, the kernel of $i$ equals the image of the map $\ad_\varphi\colon H^0\big(P\big[\fh^\C\big]\big)\to H^0\big(P\big[\fm^\C\big]\otimes K\big)$.

The projection $T_{(\bar\p_P,\varphi)} \Hh(\sG)\to\Omega^{0,1}\big(P\big[\fh^\C\big]\big)$ descends to a map
\begin{gather*}\pi\colon \ T_{[\bar\p_P,\varphi]}\Mm(\sG)\to H^1\big(P\big[\fh^\C\big]\big)\end{gather*} since $\alpha+\bar\p_Px$ defines the same cohomology class as~$\alpha$. Any representative of an element of the kernel of $\pi$ is a pair $\big(\bar\p_Px,\psi\big)$ such that $\bar\p_P\psi+\big[\bar\p_Px,\varphi\big]=0$. Any such pair is equivalent to $(0,\psi-\ad_\varphi x)$. Thus, the kernel of $\pi$ is the image of $i$.

Finally, the condition $\bar\p_P\psi+[\varphi,\alpha]=0$ implies that $\ad_\varphi(\alpha)$ is zero in the cohomology group $H^1\big(P\big[\fm^\C\big]\otimes K\big)$. Thus, the image of $\pi$ is the kernel of $\ad_\varphi$.
\end{proof}

\begin{Remark}For any Higgs bundle we have an analogous sequence which fails to be exact on the left and may or may not also fail to be exact on the right. One way to describe this is with a deformation complex (see~\cite{BiswasRamananInfitesimal}). Namely, the sheaf map $\ad_\varphi\colon P\big[\fh^\C\big]\to P\big[\fm^\C\big]\otimes K$ defines a~long exact sequence in hypercohomology
\begin{gather*}\xymatrix@R=1em@C=2em{0\ar[r]&\HH^0\big(\bar\p_P,\varphi\big)\ar[r]&H^0\big(P\big[\fh^\C\big]\big)\ar[r]^{\ad_\varphi \ \ \ }&H^0\big(P\big[\fm^\C\big]\otimes K\big)\ar[r]&\HH^1\big(\bar\p_P,\varphi\big)\ar@{->}`r/3pt [d] `/10pt[l] `^d[llll] `^r/3pt[d][dlll]\\&H^1\big(P\big[\fh^\C\big]\big)\ar[r]_{\ad_\varphi \ \ \ }&H^1\big(P\big[\fm^\C\big]\otimes K\big)\ar[r]&\HH^2\big(\bar\p_P,\varphi\big)\ar[r]&0.}\end{gather*}
In general, $\HH^0\big(\bar\p_P,\varphi\big)$ is the space of infinitesimal automorphisms of $\big(\bar\p_P,\varphi\big)$, and for stable Higgs bundles, the tangent space $T_{[\bar\p_P,\varphi]}\Mm(\sG)$ is identified with $\HH^1\big(\bar\p_P,\varphi\big)$.
\end{Remark}
\subsection[The $\C^*$-action]{The $\boldsymbol{\C^*}$-action}
There is a natural action of $\C^*$ on the $\sG$-Higgs bundle moduli space given by scaling the Higgs field
\begin{gather*}\xymatrix@R=0em{\C^*\times\Mm(\sG)\ar[r]&\Mm(\sG),\\ \big(\lambda,\big[\bar\p_P,\varphi\big]\big)\ar@{|->}[r]&\big[\bar\p_P,\lambda\varphi\big].}\end{gather*}
Note that the Hitchin fibration \eqref{eq hitchin fibration} is equivariant with respect to a weighted $\C^*$-action:
\begin{gather*}h\big(\big[\bar\p_P,\lambda\cdot\varphi\big]\big)=\big(\lambda^{m_1+1}p_1(\varphi),\dots,\lambda^{m_{\rk(\sG)}+1}p_{\rk(\sG)}(\varphi)\big).\end{gather*}
Thus, the fixed points of the $\C^*$-action are contained in the \emph{nilpotent cone} $h^{-1}(0)$.
Moreover, the properness of $h$ implies that $\lim\limits_{\lambda\to0}\big[\bar\p_P,\lambda\varphi\big]$ always exists and is a $\C^*$-fixed point.

Since we are dealing with isomorphism classes, being a $\C^*$-fixed point does not imply $\varphi=0$. Rather, it implies that there is a holomorphic gauge transformation $g_\lambda$ such that $\Ad_{g_{\lambda}^{-1}}\varphi=\lambda\varphi$ for all $\lambda\in\C^*$. For $\sSL(n,\C)$, the $\C^*$-fixed points are classified by the following proposition.
\begin{Proposition}Let $(E,\Phi)$ be a polystable $\sSL(n,\C)$-Higgs bundle. Then $(E,\Phi)$ is gauge equivalent to $(E,\lambda\Phi)$ for all $\lambda\in\C^*$ if and only if there is a holomorphic splitting $E=E_1\oplus\cdots\oplus E_\ell$ in which the Higgs field is given by
\begin{gather*}\Phi=\mtrx{0&\\\varphi_1&0\\&\ddots&\ddots\\&&\varphi_{\ell-1}&0},\end{gather*}
where $\varphi_j\colon E_j\to E_{j+1}\otimes K$ is a holomorphic bundle map.
\end{Proposition}
\begin{Remark}
We will usually represent such a fixed point schematically as
\begin{gather*}
\xymatrix{E_1\ar[r]_{\varphi_1}&E_2\ar[r]_{\varphi_2}&\cdots\ar[r]_{\varphi_{\ell-2}}&E_{\ell-1}\ar[r]_{\varphi_{\ell-1}}&E_\ell,}
\end{gather*}
where we suppress the twisting by $K$ from the notation. The moduli space of such fixed points is a special case of the moduli of holomorphic chains.
\end{Remark}

For $\sSL(n,\C)$ we have $\sH^\C=\sSL(n,\C)$ and $\fm^\C=\fsl(n,\C)$. For $\sSL(n,\C)$-Higgs bundles fixed by the $\C^*$-action, the $\sH^\C$-bundle has a holomorphic reduction $E=E_1\oplus\cdots\oplus E_k$ to a subgroup of block diagonal matrices. Such a reduction gives a $\Z$-grading on the bundle $\Ee\big[\fm^\C\big]=\End(E)=\bigoplus_j\End(E)_h$, where
\begin{gather*}\End(E)_j=\bigoplus\limits_{b-a=j}\Hom(E_a,E_b).\end{gather*}
Moreover, with respect to this $\Z$-grading we have $\Phi\in H^0(\End(E)_{1}\otimes K)$. The characterization of $\sG$-Higgs bundles fixed by the $\C^*$-action is given by the following proposition.
\begin{Proposition}
\label{Prop G Higgs fixed point}A polystable $\sG$-Higgs bundle $(\Pp,\varphi)$ defines a fixed point of the $\C^*$-action on~$\Mm(\sG)$ if and only if
\begin{enumerate}\itemsep=0pt
\item[$1.$] There is a $\Z$-grading $\fg^\C=\bigoplus_j\fg^\C_j=\bigoplus_j\fh_j^\C\oplus\fm^\C_j$ so that $[\fg_j,\fg_i]\subset\fg_{i+j}$.
\item[$2.$] There is a holomorphic reduction $\Pp_{\sH^\C_0}\subset\Pp$ to an $\sH^\C_0$-bundle, where $\sH^\C_0<\sH^\C$ is the Lie group with Lie algebra $\fh_0^\C$.
\item[$3.$] With respect to the decomposition $\Pp_{\sH_0^\C}\big[\fm^\C\big]=\bigoplus_j\Pp_{\sH_0^\C}\big[\fm^\C_j\big]$, we have
\begin{gather*}\varphi\in H^0\big(\Pp_{\sH^\C_0}\big[\fm_{1}^\C\big]\otimes K\big).\end{gather*}
\end{enumerate}
\end{Proposition}
\begin{Remark}
In terms of vector bundles, the $\sG$-Higgs bundles which are $\C^*$-fixed points are given by holomorphic chains with extra symmetries which reflect the symmetries of a $\sG$-Higgs bundle. For example, the $\sSL(p+q,\C)$-Higgs bundle associated to an $\sSO(p,q)$-Higgs bundle $(V,W,\eta)$ is given by $\Big(V\oplus W,\smtrx{0&\eta^\dagger\\\eta&0}\Big)$, so the associated fixed points are direct sums of holomorphic chains of the form
\begin{gather*}\xymatrix{V_{r}\ar[r]^{\eta_r}&W_{r-1}\ar[r]^{\eta_{2-r}^\dagger}&\cdots\ar[r]^{\eta_{r-2}}&W_{1-r}\ar[r]^{\eta_r^\dagger}&V_{-r}}\end{gather*}
and
\begin{gather*}\xymatrix{W_{s}\ar[r]^{\eta_{1-s}^\dagger}&V_{s-1}\ar[r]^{\eta_{s-1}}&\cdots\ar[r]^{\eta_{s-1}^\dagger}&V_{1-s}\ar[r]^{\eta_{1-s}}&W_{-s}.}\end{gather*}
Here $r$ and $s$ are half integers and the additional symmetry on the grading comes from the orthogonal structure. Namely, the quadratic forms give isomorphisms $W_{-j}\cong W_j^*$ and $V_{-j}\cong V_j^*$.
\end{Remark}

\subsection{Critical points of a Morse--Bott function}
So far we have not used the full power of the nonabelian Hodge correspondence. Since we have a special metric associated to each polystable Higgs bundle, we can take the~$\mathrm{L}^2$-norm of the Higgs field. Namely, consider the nonnegative function $f\colon \Mm(\sG)\to\R$ defined by
\begin{gather}
\label{eq Morse-Bott functions} f\big(\big[\bar\p_P,\varphi\big]\big)=\int_X|\varphi|^2,
\end{gather}
where the norm $|\varphi|$ is taken with respect to the Hermitian metric associated to $\big(\bar\p_P,\varphi\big)$ from the nonabelian Hodge correspondence (see Remark \ref{herm metric remark}).
\begin{Remark}
Note that the function $f\big(\bar\p_P,\varphi\big)=0$ if and only if $\varphi=0$. Equivalently, the global minima of $f$ are given by points in the moduli of polystable $\sH^\C$-bundles $\Mm(\sH)\subset\Mm(\sG)$.
\end{Remark}
For $\lambda\in\sU(1)$, the metrics from the nonabelian Hodge correspondence associated $\big(\bar\p_P,\varphi\big)$ and $\big(\bar\p_P,\lambda\varphi\big)$ are the same.\footnote{This is not necessarily true when $\lambda\in\C^*$.} Thus, the function $f$ is $\sU(1)$-invariant. Moreover, in \cite[Section 8]{liegroupsteichmuller}, Hitchin showed that the $\sU(1)$-action is Hamiltonian with respect to the symplectic structure $\omega_I$ from Remark \ref{Rem: symplectic form}, and that the function $f$ is a moment map for this action. That is,
\begin{gather*}\mathrm{grad}(f)=IX,\end{gather*}
where $X$ is the vector field generating the $\sU(1)$-action. This implies that $f$ is a Morse--Bott function on the smooth locus of $\Mm(\sSL(n,\C))$ and critical submanifolds of $f$ are exactly the components of the fixed point set of the $\sU(1)$-action. In fact, Hitchin's arguments also hold for the moduli space~$\Mm(\sG)$.

Since the moduli space $\Mm(\sG)$ is usually not smooth, we cannot use the full power of Morse theory to do things like compute the cohomology ring. However, using Uhlenbeck compactness Hitchin showed that the function $f$ proper \cite{selfduality} even on the singular locus. Hence, $f$ attains a minimum on every closed subset. In particular, we have the following upper bound on the number of components
\begin{gather*}|\pi_0(\Mm(\sG))|\leq |\pi_0(\{\text{local minima of $f$}\})|.\end{gather*}

There are three subvarieties which intersect at a $\C^*$-fixed point $\big[\bar\p_P,\varphi\big]$
\begin{enumerate}\itemsep=0pt
\item[1)] $W^s\big({\big[\bar\p_P,\varphi\big]}\big)=\big\{\big[\bar\p_P',\vartheta\big]\,|\,\lim\limits_{\lambda\to0}\big[\bar\p_P',\lambda\vartheta\big]=\big[\bar\p_P,\varphi\big]\big\}$,
\item[2)] $W^u\big({\big[\bar\p_P,\varphi\big]}\big)=\big\{\big[\bar\p_P',\vartheta\big]\,|\,\lim\limits_{\lambda\to\infty}\big[\bar\p_P',\lambda\vartheta\big]=\big[\bar\p_P,\varphi\big]\big\}$,
\item[3)] $W^0\big({\big[\bar\p_P,\varphi\big]}\big)$ the connected component of the fixed point locus containing $\big[\bar\p_P,\varphi\big]$.
\end{enumerate}
For a stable fixed point $\big[\bar\p_P,\varphi\big]$, these are exactly the stable, unstable and critical submanifolds of the Morse--Bott function $f$ at the critical point $\big[\bar\p_P,\varphi\big]$. A fixed point $\big[\bar\p_P,\varphi\big]$ is thus a local minima of $f$ if and only if $W^u\big({\big[\bar\p_P,\varphi\big]}\big)=\big\{\big[\bar\p_P,\varphi\big]\big\}$.
\begin{Remark}\label{remark inequality for components}Recall from before that we have a lower bound on $\pi_0(\Mm(\sG))$ given by $\pi_0(\Mm(\sH))$. Moreover, when $\sH$ is semisimple we have $\pi_0(\Mm(\sH))$ is in bijective correspondence with topological $\sH$-bundles $\Bb_H(X)$. This gives the following inequality
\begin{gather*}|\pi_0(\Mm(\sH))|\leq |\pi_0(\Mm(\sG))|\leq |\pi_0(\{\text{local minima of $f$}\})|.\end{gather*}
\end{Remark}
We now have a strategy for counting the components of the character variety $\Xx(\sG)$. Namely we should classify local minima of~$f$ and show that every component of local minima defines a~component of the moduli space. In particular, if the only local minima have $\varphi=0$, then every Higgs bundle can be reduced to $\sH$, and if a component of~$\Mm(\sG)$ has the property that the Higgs field can never be deformed to zero, then no representation in the associated component of the character variety can be deformed to a compact representation.

\subsection{Local minima criterion}
We first describe how for fixed points of the $\C^*$-action we get a decomposition of the tangent space into weight spaces. Recall from Proposition~\ref{Prop G Higgs fixed point} that associated to a polystable $\sG$-Higgs bundle $(\Pp,\varphi)$ fixed by the $\C^*$-action there is a $\Z$-grading $\fg^\C=\bigoplus_j\fh^\C_j\oplus\fm^\C_j$ and a holomorphic structure group reduction~$\Pp_{\sH^\C_0}$ so that $\varphi\in H^0\big(\Pp_{\sH^\C_0}\big[\fm_1^\C\big]\otimes K\big)$.

For such a fixed point, the map $\ad_\varphi\colon \Pp\big[\fh^\C\big]\to\Pp\big[\fm^\C\big]\otimes K$ defines a map
\begin{gather*}\ad_{\varphi}\colon \ \Pp_{\sH^\C_0}\big[\fh^\C_j\big]\to\Pp_{\sH^\C_0}\big[\fm^\C_{j+1}\big]\otimes K.\end{gather*}
For stable fixed points this gives a decomposition of the exact sequence \eqref{eq exact sequence}, that is, for all $j$ we have
\begin{gather*}
\xymatrix@R=1em@C=2em{0\ar[r]&H^0\big(\Pp_{\sH^\C_0}\big[\fh_j^\C\big]\big)\ar[r]^{\ad_\varphi \ \ \ }&H^0\big(\Pp_{\sH^\C_0}\big[\fm_{j+1}^\C\big]\otimes K\big)\ar[r]^{\ \ i}&T^j_{[\bar\p_P,\varphi]}\Mm(\sG)\ar@{->}`r/3pt [d] `/10pt[l] `^d[lll] `^r/3pt[d][dll]_{\pi\ \ \ \ \ \ \ \ \ }\\&H^1\big(\Pp_{\sH^\C_0}\big[\fh_j^\C\big]\big)\ar[r]_{\ad_\varphi \ \ \ }&H^1\big(\Pp_{\sH^\C_0}\big[\fm_{j+1}^\C\big]\otimes K\big),&}
\end{gather*}
where, similarly to \eqref{eq tangent space}, $T^j_{[\bar\p_P,\varphi]}\Mm(\sG)$ is defined by
\begin{gather*}T^j_{[\bar\p_P,\varphi]}\Mm(\sG)=\dfrac{\big\{(\alpha,\psi)\in\Omega^{0,1}\big(\Pp_{\sH^\C_0} \big[\fh^\C_j\big]\big)\oplus\Omega^{1,0}\big(\Pp_{\sH^\C_0}\big[\fm^\C_{j+1}\big]\big)\,|\,\bar\p_P\psi+[\alpha,\varphi]=0\big\}}{\big\{\big(\bar\p_Px,[\varphi x]\big)\in\Omega^{0,1}\big(\Pp_{\sH^\C_0}\big[\fh^\C_j\big]\big)\oplus\Omega^{1,0}\big(\Pp_{\sH^\C_0}\big[\fm^\C_{j+1}\big]\big)\,|\,x\in \Omega^0\big(\Pp_{\sH^\C_0}\big[\fh_{j}^\C\big]\big)\big\}}. \end{gather*}

 The following result was proven for $\sSL(n,\C)$ by Hitchin in \cite{liegroupsteichmuller}, the general case follows from arguments analogous to the Morse--Bott function's index computation of Hitchin in \cite[Section~8]{liegroupsteichmuller}.
\begin{Theorem}[Hitchin \cite{selfduality, liegroupsteichmuller}] Let $f\colon \Mm(\sG)\to \R$ be the Morse--Bott function from~\eqref{eq Morse-Bott functions}. For a stable $\sG$-Higgs bundle we have the following:
\begin{itemize}\itemsep=0pt
\item $\big[\bar\p_P,\varphi\big]$ is a fixed point of the $\C^*$-action if and only if it is a critical point of the func\-tion~$f$,
\item $T_{[\bar\p_P,\varphi]}W^u\big(\big[\bar\p_P,\varphi\big]\big)=\bigoplus\limits_{j>0}T^j_{[\bar\p_P,\varphi]}\Mm(\sG)$,
\item $T_{[\bar\p_P,\varphi]}W^s\big(\big[\bar\p_P,\varphi\big]\big)=\bigoplus\limits_{j<0}T^j_{[\bar\p_P,\varphi]}\Mm(\sG)$,
\item $T_{[\bar\p_P,\varphi]}W^0\big({\big[\bar\p_P,\varphi\big]}\big)=T^0_{[\bar\p_P,\varphi]}\Mm(\sG)$.
\end{itemize}
\end{Theorem}

\begin{Corollary}A stable fixed point $\big[\bar\p_P,\varphi\big]\in\Mm(\sG)$ is a local minimum of $f$ if and only if $\bigoplus\limits_{j>0}T^j_{[\bar\p_P,\varphi]}\Mm(\sG)=0$.
\end{Corollary}

Using the sequence \eqref{eq exact sequence}, if $\ad_\varphi\colon \Pp\big[\fh^\C_j\big]\to\Pp[\fm_{j+1}]\otimes K$ is an isomorphism for all $j>0$, then $T^j_{[\bar\p_P,\varphi]}\Mm(\sG)=0$ for all $j>0$ and we are at a local minimum of the $f$. In fact the converse holds as well (see \cite[Section 3.4]{BGGHomotopyGroups}), and we have a classification of stable local minima of the Morse--Bott function $f$.
\begin{Proposition}\label{prop stable min}A stable $\sG$-Higgs bundle $\big(\bar\p_P,\varphi\big)$ which is a $\C^*$-fixed point is a local minima of the function~$f$ from~\eqref{eq Morse-Bott functions} if and only if
\begin{gather*}\ad_{\varphi}\colon \ \Pp\big[\fh^\C_j\big]\to\Pp[\fm_{j+1}]\otimes K\end{gather*}
is an isomorphism for all $j>0$.
\end{Proposition}

\subsection{Some component results}\label{section some comp results}
We have now developed necessary tools to show the map $\tau\colon \Xx(\Gamma,\sG)\to\Bb_\sG(S)$ from \eqref{eq top invariant map} is injective. In fact, the proof is very simple with the above set up.
\begin{Theorem}[Garc\'{\i}a-Prada and Oliveira \cite{Oliveira_GarciaPrada_2016}] Let $\sG$ be a complex reductive Lie group with maximal compact subgroup~$\sH$. Then there is a bijection between the components of the moduli space of polystable $\sG$-Higgs bundles and the moduli space of polystable $\sG$-bundles:
\begin{gather*}\pi_0(\Mm(\sG))=\pi_0(\Mm(\sH)).\end{gather*}
\end{Theorem}
\begin{proof}
By the above discussion and Remark~\ref{remark inequality for components}, it suffices to show that a polystable $\sG$-Higgs bundle $[\Pp,\varphi]$ is a local minima of the Morse--Bott function~$f$ from~\eqref{eq Morse-Bott functions} if and only if $\varphi=0$. Let $[\Pp,\varphi]$ be a local minima of $f$. Since $[\Pp,\varphi]$ is a $\C^*$-fixed point, by Proposition~\ref{Prop G Higgs fixed point} there is a~$\Z$-grading $\fg^\C=\bigoplus_j\fh^\C_j\oplus\fm^\C_j$ so that $(\Pp,\varphi)$ is isomorphic to $(\Pp_{\sH^\C_o},\varphi)$ with $\varphi\in H^0\big(\Pp_{\sH^\C_0}\big[\fm^\C_1\big]\otimes K\big)$.

First suppose $\big[\bar\p_P,\varphi\big]$ is a stable local minima of $f$. Then by Proposition~\ref{prop stable min} we have
\begin{gather*}\ad_\varphi\colon \ \Pp_{\sH^\C_0}\big[\fh^\C_j\big]\to \Pp_{\sH^\C_0}\big[\fm^\C_{j+1}\big]\otimes K\end{gather*}
is an isomorphism for all $j>0$. But, since $\fg$ is complex, we have $\fh^\C\cong\fm^\C$, and the only way $\Pp_{\sH^\C_0}\big[\fh^\C_j\big]$ can be isomorphic to $\Pp_{\sH^\C_0}\big[\fm^\C_{j+1}\big]\otimes K$ for all $j>0$ is for $\Pp_{\sH^\C_0}\big[\fh^\C_j\big]=0$ for all $j>0$. In this case we have $\varphi=0$.

To rule out strictly polystable minima with nonzero Higgs field, we note that a $\sG$-Higgs bundle which is strictly polystable has a holomorphic reduction to a Levi factor $\sL$ of a parabolic subgroup of $\sG$ which is stable as a $\sL$-Higgs bundle. Now repeat the above argument for the moduli space~$\Mm(\sL)$.
\end{proof}

As a immediate corollary we have the following.
\begin{Corollary}If $\sG$ is a complex reductive Lie group, then every polystable $\sG$-Higgs bundle $\big[\bar\p_P,\varphi\big]$ can be continuously deformed to a polystable $\sG$-bundle, i.e., a polystable $\sG$-Higgs bundle with zero Higgs field.
\end{Corollary}
Using the nonabelian Hodge correspondence, Theorem~\ref{thm tau injective complex} now follows as a corollary of the above theorem.
\begin{Corollary}
If $\sG$ is a complex reductive Lie group with maximal compact~$\sH$, then the map $\tau\colon \pi_0(\Xx(\Gamma,\sG))\to\Bb_\sG(S)$ from~\eqref{eq top invariant map} is injective. In particular, every representation $\rho\colon \Gamma\to\sG$ can be deformed to a compact representation $\Gamma\to\sH\hookrightarrow \sG$.
\end{Corollary}

Using the methods described above, Hitchin gave a complete component count of \linebreak $\Mm(\sP\sSL(n,\R))$. The proof idea is to first classify the stable local minima using Proposition~\ref{Prop G Higgs fixed point}, then construct explicit deformations of strictly polystable fixed points with nonzero Higgs field which decreases the value of $f$.

\begin{Theorem}[Hitchin \cite{liegroupsteichmuller}] \label{thm min PSLn}For $n>2$, the only local minima of the Morse--Bott function~\eqref{eq Morse-Bott functions} on $\Mm(\sP\sSL(n,\R))$ are $\varphi=0$ and the image of 0 in the Hitchin section.
\end{Theorem}
We thus have the following corollary.
\begin{Corollary}
\begin{gather*}|\pi_0(\Mm(\sP\sSL(n,\R))|=\begin{cases}
4g-3& n=2,\\
3&\text{$n$-odd,}\\
6&\text{$n>2$ and even}.
\end{cases}\end{gather*}
\end{Corollary}
\begin{proof}
For $n=2$ the component count is the same as for $\Mm(\sSO_0(1,2))$. For $n>2$ and odd we have $\Bb_{\sP\sSL(n,\R)}=\Z_2$ and there is only one Hitchin component. This gives three components. For $n>2$ and $n$-even we have $\Bb_{\sP\sSL(n,\R)}$ is isomorphic to $\Z_2\times \Z_2$ or $\Z_4$ depending on the parity of $\frac{n}{2}$ (see \eqref{eq top PSLn bundles}). Moreover, there are two Hitchin components by Remark \ref{Remark 2 Hitchin comp PSL2n}, this gives six components. By Theorem \ref{thm min PSLn} there are no other components.
\end{proof}

For the character variety $\Xx(\Gamma,\sP\sSL(n,\R))$ we of course have the same count.
\begin{Corollary}
\begin{gather*}|\pi_0(\Xx(\Gamma,\sP\sSL(n,\R))|=\begin{cases}
4g-3& n=2,\\
3&\text{$n$-odd,}\\
6&\text{$n>2$ and even}.
\end{cases}\end{gather*}
\end{Corollary}

Theorem \ref{thm min PSLn} also gives a dichotomy for deformations of representations into $\sP\sSL(n,\R)$, namely for $n>2$ the components of $\Xx(\Gamma,\sP\sSL(n,\R))$ are either deformations spaces of compact representations or deformation spaces of special Fuchsian representations.
\begin{Corollary}
For each $n>2$ and each $\rho\in\Xx(\Gamma,\sP\sSL(n,\R))$, exactly one of the following holds
\begin{itemize}\itemsep=0pt
\item $\rho$ can be deformed to a compact representation
\begin{gather*}\Gamma\to\sP\sSO(n)\hookrightarrow \sP\sSL(n,\R),\end{gather*}
\item $\rho$ can be deformed to a representation \begin{gather*}\xymatrix{\Gamma\ar[r]_{\rho_{\rm Fuch}\ \ \ \ }&\sP\sSL(2,\R)\ar[r]_{\iota_{\rm pr}}&\sP\sSL(n,\R),}\end{gather*} where $\rho_{\rm Fuch}\in\Fuch(\Gamma)$ is a Fuchsian representation and $\iota_{\rm pr}$ is the principal embedding from~\eqref{eq principal embedding}.
\end{itemize}
\end{Corollary}
\section[$\sSO(p,q)$-Higgs bundles]{$\boldsymbol{\sSO(p,q)}$-Higgs bundles}\label{section so(p,q)}
We now apply the techniques of the previous section to understand the components of the $\sSO(p,q)$-character variety $\Xx(\Gamma,\sSO(p,q))$. In her thesis \cite{MartaThesis}, Aparicio-Arroyo discovered that the Higgs bundle moduli space $\Mm(\sSO(p,q))$ has stable local minima of the Morse--Bott function~\eqref{eq Morse-Bott functions} with nonzero Higgs field and which do not arise from the Hitchin section. This was done by classifying stable $\sSO(p,q)$-Higgs bundles which are fixed points of the $\C^*$-action and satisfied Proposition \ref{prop stable min}. Due to the potential singularities, these results are not strong enough to classify the components of the moduli space $\Mm(\sSO(p,q))$.

It should be noted that the methods outlined in the previous section are not the only approach to studying the question of components. For example, by examining spectral data on generic fibers of the Hitchin fibration for $\Mm(\sSO(p+q,\C))$, Schaposnik and Baraglia~\cite{DavidLauraCayleyLanglands} have given evidence for the existence of additional components of the moduli space $\Mm(\sSO(p,q))$.
Although these methods do not currently distinguish connected components, they provide an intriguing alternative perspective.

We start by recalling the classification of stable minima.
\begin{Remark}The case $\sSO(2,q)$ is rather special since $\sSO(2,q)$ is a group of Hermitian type. This special type of group has its own very interesting connected component results. Since we have not said much about this situation, we will only discuss the non-Hermitian case, that is, for $2<p\leq q$. For the case of $\sSO(2,q)$ we refer the reader to \cite{SOpqComponents,CollierTholozanToulisse}.
\end{Remark}

Recall the notation of $\Kk_p=K^{p-1}\oplus K^{p-3}\oplus\cdots\oplus K^{3-p}\oplus K^{1-p}$ from \eqref{eq Kk_p notation} and the map $\widehat\Psi(0,\dots,0)\colon \Kk_p\to\Kk_{p+1}\otimes K$ from \eqref{eq image Psi hat}. Denote by $\eta_0$ the following transpose
\begin{gather*}\eta_0=\widehat\Psi(0,\dots,0)^T\colon \ \Kk_p\to\Kk_{p-1}\otimes K.\end{gather*}

\begin{Theorem}[Aparicio-Arroyo \cite{MartaThesis}] \label{thm sopq stable min} Suppose $2<p\leq q$. If $[V,W,\eta]$ is a stable $\sSO(p,q)$-Higgs bundle, then $(V,W,\eta)$ defines a local minima of the Morse--Bott function \eqref{eq Morse-Bott functions} if and only if one of the following holds
\begin{enumerate}\itemsep=0pt
\item[$1)$] $\eta=0$,
\item[$2)$] there is a stable rank $q-p+1$ orthogonal bundle $W_0$ with determinant bundle $I=\det(W_0)$ such that
\begin{gather*}(V,W,\eta)=\left(\Kk_p\otimes I,(\Kk_{p-1}\otimes I)\oplus W_0, \mtrx{\eta_0\\0}\colon V\to W\otimes K\right),\end{gather*}
\item[$3)$] $q=p+1$ and there is a line bundle $M\in\Pic^d(X)$ with $d\in(0,p(2g-2)]$ and $\mu\in H^0\big(M^{-1}K^p\big)\setminus\{0\}$ so that
\begin{gather*}(V,W,\eta)=\left(\Kk_p,M\oplus \Kk_{p-1}\oplus M^{-1}, \mtrx{0\\\eta_0\\\eta_\mu}\colon V\to W\otimes K\right),\end{gather*}
where $\eta_\mu=\mtrx{0&\cdots&0&\mu}\colon \Kk_p\to M^{-1}K$.
\end{enumerate}
\end{Theorem}
\begin{Remark}Note that in case three of the above theorem when $\deg(M)= p(2g-2)$ the existence of a nonzero section of $M^{-1}K^p$ implies $M=K^p$. In this case, the minima is the minima in the $\sSO(p,p+1)$-Hitchin component defined by $\widehat\Psi(0,\dots,0)$ from~\eqref{eq image Psi hat}.
\end{Remark}

In \cite{SOpqComponents} all of the local minima are classified. The result basically says that the only minima not accounted for in Theorem~\ref{thm sopq stable min} arise from polystable Higgs bundles with zero Higgs field and from Higgs bundles similar to the second case of Theorem~\ref{thm sopq stable min}, but where the bundle~$W_0$ is allowed to be strictly polystable.

\begin{Theorem}\label{Thm local min SOpq}Assume $2<p\leq q$ and $(V,W,\eta)$ is a polystable $\sSO(p,q)$-Higgs bundle. Then $(V,W,\eta)$ defines a local minima of the Morse--Bott function~\eqref{eq Morse-Bott functions} if and only if
\begin{enumerate}\itemsep=0pt
\item[$1)$] $\eta=0$,
\item[$2)$] there is a polystable rank $q-p+1$ orthogonal bundle $W_0$ with determinant bundle $I=\det(W_0)$ such that
\begin{gather*}(V,W,\eta)=\left(\Kk_p\otimes I,(\Kk_{p-1}\otimes I)\oplus W_0, \mtrx{\eta_0\\0}\colon V\to W\otimes K\right),\end{gather*}
\item[$3)$] $q=p+1$ and there is a line bundle $M\in\Pic^d(X)$ with $d\in(0,p(2g-2)]$ and $\mu\in H^0\big(M^{-1}K^p\big)\setminus\{0\}$ so that
\begin{gather*}(V,W,\eta)=\left(\Kk_p,M\oplus \Kk_{p-1}\oplus M^{-1}, \mtrx{0\\\eta_0\\\eta_\mu}\colon V\to W\otimes K\right),\end{gather*}
where $\eta_\mu=\mtrx{0&\cdots&0&\mu}\colon \Kk_p\to M^{-1}K$.
\end{enumerate}
\end{Theorem}

To show that each of the above local minima types defines a connected component of moduli space $\Mm(\sSO(p,q))$, we define a map $\Theta_{p,q}$ from a parameter space into $\Mm(\sSO(p,q))$ so that
\begin{enumerate}\itemsep=0pt
 \item[1)] $\Theta_{p,q}$ is a homeomorphism onto its image,
 \item[2)] the image of $\Theta_{p,q}$ is open and closed,
 \item[3)] each component of the image of $\Theta_{p,q}$ contains exactly one connected component of the local minima with $\eta\neq0$ from Theorem~\ref{Thm local min SOpq}.
 \end{enumerate}

The connected components of the local minima of type 2 in the above theorem are determined by the number of components of polystable $\sO(q-p+1,\C)$-bundles, that is, the components of $\Mm(\sO(q-p+1))$. For $q>p+1$ the group $\sO(q-p+1)$ is simple thus, by Theorem~\ref{thm tau injective compact}
\begin{gather*}|\pi_0(\Mm(\sO(q-p+1))|=|\Bb_{\sO(q-p+1)}(S)|= 2^{2g+1}.\end{gather*}
When $q=p+1$, we have $\Mm(\sO(q-p+1))=\Mm(\sO(2))$ and the number of connected components is $2^{2g+1}-1$. Combining these with the $p(2g-2)$ components of local minima of type 3 in the above theorem gives
$2^{2g+1}-1+p(2g-2)$ connected components of local minima in $\Mm(\sSO(p,p+1))$ with $\eta\neq0$.
Finally, when $q=p$ we have $q-p+1=1$ and thus there are $2^{2g+1}$ components of local minima with $\eta\neq0$. In this case all such minima define Hitchin components.

Combined with the $2^{2g+1}$-components of $\Mm(\sS(\sO(p)\times \sO(q))$, for $2<p\leq q$ the following theorem of~\cite{SOpqComponents} establishes the component count of~$\Mm(\sSO(p,q))$.

\begin{Theorem}\label{component sopq}For $2<p\leq q$, we have
\begin{gather*}|\pi_0(\Mm(\sSO(p,q)))|=2^{2g+1}+\begin{cases}
2^{2g+1}&q=p,\\2^{2g+1}-1+p(2g-2)&q=p+1,\\2^{2g+1}&\text{else}.
\end{cases}\end{gather*}
\end{Theorem}

To sketch the idea of the proof of the above theorem, we need to slightly generalize our notion of Higgs bundles.
\begin{Definition}A $K^p$-twisted $\sG$-Higgs bundle is a pair $(\Pp,\varphi)$ where\samepage
\begin{itemize}\itemsep=0pt
\item $\Pp\to X$ is a holomorphic $\sH^\C$-bundle,
\item $\varphi$ is a holomorphic section of the associated bundle $\Pp\big[\fm^\C\big]\otimes K^p$.
\end{itemize}
\end{Definition}

The notions of stability for $K^p$-twisted Higgs bundles are defined similarly to the notions of stability for regular Higgs bundles, i.e., for $K^1$-twisted Higgs bundles. We will denote the space of polystable $K^p$-twisted $\sG$-Higgs bundles and the resulting moduli space by
\begin{gather*} \Hh^{\rm ps}_{K^p}(\sG)\qquad \text{and}\qquad \Mm_{K^p}(\sG)=\Hh^{\rm ps}_{K^p}(\sG)/\Gg_{\sH^\C}. \end{gather*}

Recall from Section \ref{section SO1q and Hitchin comp} that an $\sSO(1,n)$-Higgs bundle is given by a triple $(V,W,\eta)=(\Lambda^nW_0$, $W_0,\eta)$, where $\eta\in H^0(\Lambda^qW_0\otimes W_0\otimes K)$. The map $\eta$ can be interpreted as a holomorphic bundle map $\eta\colon \Lambda^qW_0\to W_0\otimes K$. Similarly, a $K^p$-twisted $\sSO(1,n)$-Higgs bundle is given by a triple $(\Lambda^nW_0,W_0,\eta_p)$ where $\eta_p\in H^0(\Lambda^qW_0\otimes W_0\otimes K^p)$, which we may interpret at as holomorphic bundle map $\eta_p\colon K^{1-p}\to W_0\otimes K$.

Recall the definition of the map $\widehat\Psi\colon \bigoplus\limits_{j=1}^p H^0\big(K^{2j}\big)\to \Hh(\sSO(p,p+1))$ from~\eqref{eq image Psi hat}. Taking an transpose of this map defines the $\sSO(p+1,p)$-Hitchin component. For our applications we need the map for defining the $\sSO(p,p-1)$-Hitchin component. We call this map $\widehat\Psi\colon \bigoplus\limits_{j=1}^{p-1} H^0\big(K^{2j}\big)\to \Hh(\sSO(p,p-1))$ as well, explicitly it is given by
\begin{gather*}\widehat\Psi(q_2,\dots,q_{2p-2})=\left(\Kk_{p},Q_p,\Kk_{p-1},Q_{p-1},\mtrx{1&q_2&\dots&q_{2p-2}\\&\ddots&\ddots&\vdots\\&&1&q_2}\colon \Kk_{p}\to \Kk_{p-1}\otimes K\right).\end{gather*}

Consider the map
\begin{gather}
\label{eq Thetapq}\widehat\Theta_{p,q}\colon \ \Hh^{\rm ps}_{K^p}(\sSO(1,q-p+q))\times \bigoplus\limits_{j=1}^{p-1}H^0\big(K^{2j}\big)\longrightarrow \Hh(\sSO(p,q))
\end{gather}
defined by
\begin{gather*}\widehat\Theta_{p,q}(W_0,\eta_p,q_2,\dots,q_{2p-2})=\big(I\otimes\Kk_p,(I\otimes \Kk_{p-1})\oplus W_0,\big(\begin{matrix} \widehat\Psi(q_2,\dots,q_{2p-2})&\hat\eta_p\end{matrix}\big)\big),\end{gather*}
where $I=\Lambda^{q-p+1}W_0$ and
\begin{gather*}\hat\eta_p=\mtrx{0&\cdots&0&\eta_p}\colon \ I\otimes\big(K^{p-1}\oplus K^{p-3}\oplus\cdots\oplus K^{3-p}\oplus K^{1-p}\big)\to W_0\otimes K.\end{gather*}
\begin{Remark}Note that $\Lambda^p(I\otimes \Kk_{p})=I^p$ and $\Lambda^{q}\big(\big(I\otimes K^{p-1}\big)\oplus W_0\big)=I^p$, so this indeed defines an $\sSO(p,q)$-Higgs bundle.
\end{Remark}
\begin{Remark}\label{Rem 0 Kp higgs field}Note also that for $W_0$ a polystable orthogonal bundle of rank $q-p+1$, we can take $\eta_p=0$. In this case the image of $\widehat\Theta_{p,q}(W_0,0,0,\dots,0)$ is given by
\begin{gather*}\big(\Lambda^{q-p+1}(W_0)\otimes\Kk_p,\big(\Lambda^{q-p+1}(W_0)\otimes \Kk_{p-1}\big)\oplus W_0,\big(\begin{matrix}\widehat\Psi(q_2,\dots,q_{2p-2})&0\end{matrix}\big)\big).\end{gather*}
In particular, the $\sSO(p,q)$-Higgs bundle is a direct sum of an $\sSO(p,p-1)$-Higgs bundle in the Hitchin component (twisted by an $\sO(1,\C)$-bundle) with a polystable $\sO(q-p+1)$-Higgs bundle.
\end{Remark}
\begin{Theorem}[\cite{SOpqComponents}]\label{thm calyey com}
For $2<p\leq q$, the map $\widehat\Theta_{p,q}$ from \eqref{eq Thetapq} induces a map
\begin{gather*}\Theta_{p,q}\colon \ \Mm_{K^p}(\sSO(1,q-p+1))\times\bigoplus\limits_{j=1}^{p-1}H^0\big(K^{2j}\big)\longrightarrow \Mm(\sSO(p,q)),\end{gather*}
which is a homeomorphism onto its image. Moreover, the image of $\Theta_{p,q}$ is open and closed.
\end{Theorem}
The proof has four steps.
\begin{enumerate}\itemsep=0pt
\item Well defined: show the $\sSO(p,q)$-Higgs bundles in the image of $\Theta_{p,q}$ are polystable.
\item Injectivity: Every $\sS(\sO(1,\C)\times \sO(q-p+1,\C))$ gauge transformation of a $K^p$-twisted $\sSO(1,q-p+1)$-Higgs bundle induces a unique $\sS(\sO(p,\C)\times \sO(q,\C))$ gauge transformation preserving the image of $\widehat\Theta_{p,q}$.
\item Openness of image: This is the most difficult and technical step. We first note that \begin{gather*}\dim\left(\Mm_{K^p}(\sSO(1,q-p+1))\times\bigoplus\limits_{j=1}^{p-1}H^0\big(K^{2j}\big)\right)=\dim(\Mm(\sSO(p,q)),\end{gather*} then analyze the local structure of the singularities of the image of $\Theta_{p,q}$.
\item Closedness of the image: Use properness of the Hitchin fibration, this is analogous to the proof of closedness of the Hitchin section.
\end{enumerate}

To see that the component count of Theorem \ref{component sopq} is a corollary of Theorems~\ref{Thm local min SOpq} and~\ref{thm calyey com} we prove the following proposition.
\begin{Proposition}
 For all $p>1$ the component count of $\Mm_{K^p}(\sSO(1,n))$ is given by
 \begin{gather*}\pi_0(\Mm_{K^p}(\sSO(1,n)))=\begin{cases}
 2^{2g}&n=1,\\2^{2g+1}-1+p(2g-2)&n=2,\\2^{2g+1}&n>2.
 \end{cases}\end{gather*}
 \end{Proposition}
\begin{Remark}\label{rem comp K^p twisted}
The proof of the $n=1$ and $n>2$ are by direct computation, namely we show that every fixed point of the $\C^*$-action can be deformed to on with zero Higgs field. The additional $p(2g-2)$ components in the $n=2$ case are analogous to the components in Theorem~\ref{thm so12 comp}. In particular, the Higgs field in these components is \emph{never zero} and these components are parameterized by certain vector bundles over an appropriate symmetric product of the surface.
\end{Remark}
As a direct corollary of the component count for $\Mm(\sSO(p,q))$ we have the following component count of the character variety $\Xx(\Gamma,\sSO(p,q))$.
\begin{Corollary}
For $2<p\leq q$, the component count of the character variety $\Xx(\Gamma,\sSO(p,q))$ is given by\footnote{When $p=q$, the extra $2^{2g}$ components arise from switching $V$ and $W$ in the image of $\Theta_{p,p}$.}
\begin{gather*}|\pi_0(\Xx(\Gamma,\sSO(p,q)))|=2^{2g+1}+\begin{cases}
2^{2g+1}&q=p,\\2^{2g+1}-1+p(2g-2)&q=p+1,\\2^{2g+1}&\text{else}.
\end{cases}\end{gather*}
\end{Corollary}

Combining Remarks \ref{Rem 0 Kp higgs field} and \ref{rem comp K^p twisted} it follows that if $2<p<q-1$, then every Higgs bundle in the image of $\Theta_{p,q}$ can be deformed to the direct sum of a polystable orthogonal bundle $W_0$ and a Higgs bundle in the $\sSO(p,p-1)$-Hitchin component twisted by the determinant of $W_0$. Applying the nonabelian Hodge correspondence to this statement gives a dichotomy for the character variety $\Xx(\Gamma,\sSO(p,q))$ when $q>p+1$ which is analogous to Corollary~\ref{cor dichotomy hitchin}.
\begin{Theorem}\label{thm dichotomy pq}
Suppose $2<p<q-1$. If $\rho\in\Xx(\Gamma,\sSO(p,q))$, then there is a dichotomy: either~$\rho$ can be deformed to compact representation or $\rho$ can be deformed to a Fuchsian representation of the form
\begin{gather}\label{eq pos fuch}
(\iota_{p,q}\circ\iota_{\rm pr}\circ\rho_{\rm Fuch})\otimes \det(\alpha)\oplus\alpha,
\end{gather}
where
\begin{itemize}\itemsep=0pt
\item $\rho_{\rm Fuch}\colon \Gamma\to\sP\sSL(2,\R)$ is a Fuchsian representation,
\item $\iota_{\rm pr}\colon \sP\sSL(2,\R)\to\sSO(p,p-1)$ is the principal embedding,
\item $\iota_{p,q}\colon \sSO(p,p-1)\to\sSO(p,q)$ is the embedding given by \eqref{eq pq embedd},
\item $\alpha\colon \Gamma\to\sO(q-p+1)$ is a compact representation.
\end{itemize}
In particular, every component of $\Xx(\Gamma,\sSO(p,q))$ is either the deformation space of compact representations or the deformation space of certain Fuchsian representations.
\end{Theorem}
\begin{Remark}
Recently, Guichard and Wienhard have developed a notion called positivity which conjecturally characterizes components of the character variety which deserve the name ``higher Teichm\"uller spaces'' \cite{PosRepsGWPROCEEDINGS}. In this work, the classical groups which have such a positive structure are split groups, Hermitian groups (of tube type) and $\sSO(p,q)$ for $1<p<q$. For the split and Hermitian groups, positive $\sSO(p,q)$-representations are exactly the set of Hitchin representations and maximal representations respectively. Moreover, it is not hard to show that the representations in \eqref{eq pos fuch} define positive representations. Thus, if certain conjectures of Guichard--Wienhard hold, the components defined by Theorem \ref{thm calyey com} are exactly the higher Tiechm\"uller components for the group $\sSO(p,q)$.
\end{Remark}
\begin{Remark}
For the case $\sSO(p,p+1)$ there is a trichotomy, since their are $p(2g-2)-1$ components which cannot be deformed to compact representations and cannot be deformed to Fuchsian representations. In \cite{CollierSOnn+1components}, the $\sSO(p,p+1)$-case is studied in detail, and it is conjectured that every representation in these $p(2g-2)-1$ components is Zariski dense.
\end{Remark}

\subsection*{Acknowledgments}
The article is roughly based on a three hour mini-course given by the author at University of Illinois at Chicago in June 2018. I would like to thank Qiongling Li for her complementary mini-course. I would also like to thank all the participants of the mini-courses for their enthusiasm and interest. The author is funded by a National Science Foundation Mathematical Sciences Postdoctoral Fellowship, NSF MSPRF no.~1604263.

\pdfbookmark[1]{References}{ref}
\LastPageEnding

\end{document}